\documentclass[11pt]{amsart}
\usepackage{amsmath}
\usepackage{amssymb,amsfonts,mathrsfs, amsthm, stmaryrd, color}
\usepackage[margin=2 cm,heightrounded=true,centering]{geometry}
\usepackage{amscd}
\usepackage{verbatim}
\usepackage{euscript}
\usepackage{multicol}
\usepackage{graphicx}
\newtheorem{theorem}{Theorem}[section]
\newtheorem{proposition}[theorem]{Proposition}
\newtheorem{lemma}[theorem]{Lemma}
\newtheorem{remark}[theorem]{Remark}

\newtheorem{definition}[theorem]{Definition}

\DeclareMathOperator{\id}{id}

\DeclareMathOperator{\Orth}{O}

\DeclareMathOperator{\diag}{diag}
\DeclareMathOperator{\Span}{Span}

\begin{document}

\title[Self-similar curve shortening flow in hyperbolic 2-space]{Self-similar curve shortening flow in hyperbolic 2-space}

\author{Eric Woolgar}
\address{Dept of Mathematical and Statistical Sciences, and Theoretical Physics Institute, University of Alberta, Edmonton, AB, Canada T6G 2G1.}
\email{ewoolgar(at)ualberta.ca}

\author{Ran Xie}
\address{School of Mathematics and Statistics, Xi'an Jiaotong University, Xi'an, PR China}
\email{dhsieh29(at)gmail.com}

\date{\today}

\begin{abstract}
\noindent We find and classify self-similar solutions of the curve shortening flow in standard hyperbolic 2-space. Together with earlier work of Halld\'orsson on curve shortening flow in the plane and Santos dos Reis and Tenenblat in the 2-sphere, this completes the classification of self-similar curve shortening flows in the constant curvature model spaces in 2-dimensions.
\end{abstract}

\maketitle
\section{Introduction}
\setcounter{equation}{0}

\noindent
Consider a smooth map $X:I\times [0,T)\to M:(x,t)\mapsto X(x,t)=:X_t(x)$, for $I$ a connected interval of the real line and $T\in (0,\infty]$. We will take $M\subset {\mathbb R}^3$ to be a surface, though the flow can also be contemplated for $M$ a manifold with arbitrary dimension. Then a curve shortening flow is a solution of the differential equation
\begin{equation}
\label{eq1.1}
\frac{\partial X}{\partial t}=\kappa_g N\ ,
\end{equation}
where $\kappa_g$ is the geodesic curvature of the curve $X_t(x)$ (for fixed $t$) and $N$ is the principal normal to the curve, with the signs chosen so that $\kappa_g N$ points to the concave side of the curve.

The theory of the curve shortening flow, often denoted simply as CSF, was developed in a number of papers in the 1980s, among them Gage \cite{Gage1, Gage2}, Gage and Hamilton \cite{GH}, and Grayson \cite{Grayson}. There is now an extensive literature, including at least one book \cite{CZ}.

Geodesics are trivial solutions of this flow since they have $\kappa_g=0$. The next simplest solutions are curves that evolve self-similarly, sometimes called soliton solutions. This class includes geodesics but also includes nontrivial examples. The well-known \emph{grim reaper} or \emph{paperclip}, whose trace at fixed $t$ is the graph $y(x)=C+\log\sec x$, is an example of a non-geodesic self-similar solution of CSF in ${\mathbb R}^2$. In a thesis in 2013, Halld\'orsson \cite{Halldorsson1, Halldorsson2} classified the self-similar solutions in ${\mathbb R}^2$. Self-similar solutions on the 2-sphere ${\mathbb S}^2$ were classified by Santos dos Reis and Tenenblat \cite{SdRT} in 2019. In the present paper, we study self-similar curve shortening flows in the remaining complete constant curvature surface, standard hyperbolic 2-space ${\mathbb H}^2$.

\begin{definition}\label{definition1.1}
Two unit speed curves $\gamma, {\tilde \gamma}:I\to (M,g)$ are \emph{congruent} if there is an isometry $\varphi$ of $(M,g)$ such that $\varphi\circ\gamma={\tilde \gamma}$. A curve shortening flow \emph{evolves by isometries} if $\gamma_t(s)$ is a unit speed curve for each $t\in J$, $0\in J \subset {\mathbb R}$, and there is a one-parameter family of isometries $\varphi_t$ such that $\varphi_t\circ\gamma=\gamma_t$; i.e., each $\gamma_t$ is congruent to every other one. A curve that evolves by isometries is also said to be \emph{self-similar} or a \emph{soliton}.
\end{definition}

We will further take solitons to be \emph{inextendible}, meaning that the soliton curve admits a unit speed parametrization $X(s)$ on an open connected interval $I\in{\mathbb R}$ such that $X(s)$ cannot be defined as a unit speed curve on any open connected interval $J\supset I$ containing the closure of $I$ as a proper subset.

Note that self-similarity of curves only makes sense on surfaces that have families of sufficiently smooth isometries. By the fundamental theorem of curves in surfaces, when $M$ is a surface we can determine whether the curves $\gamma$ and ${\tilde \gamma}$ are congruent by comparing their curvatures as functions of arclength. Using this fact, a consequence of \cite[Proposition 4.2]{Gage3} is that closed curves on surfaces cannot evolve by isometries under CSF. However they can evolve by the composition of isometries and  time-dependent rescalings; often the terms``self-similar'' and ``scaling soliton'' are employed to describe curves that evolve in this more general sense, but this paper is concerned with solitons that do not rescale in time.

This brings us to our main theorem.
\begin{theorem}\label{theorem1.2}
Let ${\mathbb M}^3$ be ${\mathbb R}^3$ equipped with the quadratic form (Minkowski metric) $\eta=\diag (1,1,-1)$ and denote the Minkowski inner product by $\langle U,W\rangle_{\eta}=\eta(U,W)$ for $U,W\in{\mathbb M}^3$. For each ${\tilde v}\in {\mathbb M}^3\setminus \{ 0 \}$ there is a $2$-parameter family of nontrivial solutions $X$ of CSF evolving by isometries in standard hyperbolic $2$-space ${\mathbb H}^2$. These soliton curves are complete, unbounded, and properly embedded, and asymptote either to a horocycle or to a geodesic. The soliton cannot asymptote to a geodesic
\begin{itemize}
\item [(i)] at both ends, or
\item [(ii)] if ${\tilde v}$ is timelike, or
\item [(iii)] if $\mu(s)=\left \langle X(s), {\tilde v} \right \rangle_{\eta}$ has a critical point (and it can have at most one critical point), or
\item [(iv)] if $\mu(s)$ has a zero (and it can have at most one zero).
\end{itemize}
\end{theorem}

While the results of this paper can be viewed as a natural outgrowth of the work reported in \cite{Halldorsson1, Halldorsson2} and \cite{SdRT}, there are important independent reasons to study the curve shortening flow and its solitons in hyperbolic space. There are two generalization of this problem to higher dimensions which are important for physics. The first generalization is the mean curvature flow of codimension one objects (hypersurfaces) flowing in standard hyperbolic $n$-space. The fixed points are minimal surfaces bounded by a curve on the boundary at conformal infinity. These surfaces play an important role in the AdS/CFT correspondence. The surface areas of the minimal surfaces are proportional to the entanglement entropy of the region $R$ enclosed by the curve on the conformal boundary. This entropy is a property of quantum states of a conformal field theory defined on the conformal boundary, and encodes uncertainty in measurements of those states within $R$ due to correlations which extend beyond $R$. One way to construct these minimal surfaces is as limits of convergent mean curvature flows. But these mean curvature flows could instead approach ``generalized fixed points'', the self-similar solutions, as limits. The problem of self-similar curve shortening flows in ${\mathbb H}^2$ is the obvious first step to complete before addressing the higher dimensional mean curvature flow version that arises in AdS/CFT physics. Our result that self-similar curves asymptote to geodesics only in limited cases, and never at both ends, may suggest that solitons of the higher dimensional problem would not meet the boundary at infinity orthogonally, and perhaps not even transversally.

The second application arises by considering dimension one flows in a spacetime of dimension $n+1$ with metric $g=-dt^2+a^2(t)h$, for $h$ a metric on a Riemannian $n$-manifold. This situation is of interest in cosmology, particularly when $h$ is any constant curvature metric. So-called \emph{cosmic strings} in this spacetime can be modelled as solutions of the wavemap equation for an embedded timelike 2-surface $X:(u,w)\mapsto (X^0,X^i)$, $i=1,\dots,n$, in spacetime. In this setting this becomes
\begin{equation}
\label{eq1.2}
\Box_{\eta} X^i +2H(t) \eta^{ab}\partial_aX^i\partial_b X^0=0\ ,
\end{equation}
where $H(t):=\frac{a'(t)}{a(t)}$, $\eta$ is the induced Lorentzian metric on the 2-surface, and $\Box_{\eta}$ is the trace of the Hessian of $\eta$ (regarded as operating on an $n$-tuple of functions $X^i$). If one takes the $w$ parameter to be the $t$-coordinate and the $u$ parameter to be an arclength $s$ along the curves $X_t(s)=X(s,t)$, then $X^0=t$ and one obtains \cite[Equation 15 with $\Delta=\frac{\partial^2}{\partial s^2}$ and assuming the second time derivative term is small enough to ignore]{Kibble}
\begin{equation}
\label{eq1.3}
\frac{\partial^2 X^i}{\partial s^2}=2H(t)\frac{\partial X^i}{\partial t}\ .
\end{equation}
For inflationary cosmology models, $H(t)$ is constant, and then \eqref{eq1.3} is a curve shortening flow. Therefore, for this problem we increase the dimension of the ambient manifold, which should now be any constant curvature space, and study curve shortening flows in it.

The manuscript is organized as follows. In Section 2, we first review the model of ${\mathbb H}^2$ which represents it as the $z\ge 1$ sheet of the hyperboloid $x^2+y^2-z^2=-1$ in the Minkowski spacetime ${\mathbb M}^3$. This model permits us straightforwardly to adapt ideas of \cite{Halldorsson2, SdRT} to our setting, which we do in Sections 2.2--2.4. Specifically, we formulate the problem as an autonomous system of ordinary differential equations.
In Section 3, we discuss certain special solutions such as geodesics and horocycles. Horocycles were discussed by Grayson \cite{Grayson} and arguably are hyperbolic space analogues of the grim reaper self-similar flow in ${\mathbb R}^2$. For completeness, we briefly discuss hypercycles, which are scaling solitons (so they evolve by a composition of isometries and rescalings). In Section 4 we analyze the autonomous system from Section 2 and prove a series of lemmata leading to the proof of Theorem \ref{theorem1.2}.

After this paper appeared in preprint form, we learned of the beautiful thesis \cite{daSilva} which independently derived similar results in great detail. This work is now described in \cite{ST}.

\subsection{Acknowledgements} We are indebted to K Tenenblat for comments on the preprint version of this paper, which helped us to improve our presentation (especially, to correct an error in Proposition 2.3), and for bringing reference \cite{daSilva} to our attention. EW is grateful to the organizers and audience of the Clark University Geometric Analysis seminar of 13 Nov 2020, where these results were presented, for their interest and insightful comments. We thank to Michael Yi Li and Yingfei Yi for organizing the 2019 International Undergraduate Summer Enrichment Programme (IUSEP) at the University of Alberta, during which this work was begun. EW is supported by NSERC Discovery Grant RGPIN--2017--04896.

\section{The hyperboloidal model of hyperbolic space}
\setcounter{equation}{0}

\subsection{Elementary properties} We review some basic facts of the hyperboloidal model, which represents hyperbolic 2-space ${\mathbb H}^2$ as the $z\ge 1$ sheet of the hyperboloid $x^2+y^2-z^2=-1$ in Minkowski $3$-space ${\mathbb M}^3:=\left ( {\mathbb R}^3, \eta\right )$ where $\eta$ is the quadratic form $\diag (1,1,-1)$. We will adopt the convention that our coordinates are enumerated as $x^i =(x^1,x^2,x^3)=(x,y,z)$ with $\eta(\partial_x,\partial_x)=\eta(\partial_y,\partial_y)=+1$ so that $\partial_x$ and $\partial_y$ are \emph{spacelike} and $\eta(\partial_z,\partial_z)=-1$ so that $\partial_z$ is \emph{timelike} in the terminology common in physics.

The hyperboloid modelling ${\mathbb H}^2$ is a \emph{spacelike hypersurface} in ${\mathbb M}^3$; i.e., any vector tangent to this surface is spacelike. However, any vector $X$ from the origin to a point on ${\mathbb H}^2$ is timelike and future-directed ($\eta(X,\partial_z)<0$), and has $\eta(X,X)=-1$. In fact, any such vector is a future-directed timelike unit normal field for the surface.

The group $G$ that preserves the quadratic form $\eta$ is the orthogonal group $\Orth(2,1)$. The proper orthochronous subgroup $G$ that preserves spatial orientation and time orientation will preserve the hyperboloid sheet $z=\sqrt{1+x^2+y^2}$ and the orientation of bases for its tangent spaces (as well as preserving the choice of future and past). This subgroup lies in the connected component of the identity in the isometry group of ${\mathbb H}^2$. The one-parameter subgroups of $G$ can be classified as compositions of \emph{boosts} in the $x$-direction
\begin{equation}
\label{eq2.1}
A_1(\zeta)=\left [ \begin{array}{ccc} \cosh\zeta&0& -\sinh\zeta\\0&1&0\\ -\sinh\zeta &0&\cosh\zeta \end{array}\right ]\ ,
\end{equation}
boosts in the $y$-direction
\begin{equation}
\label{eq2.2}
A_2(\xi)= \left [ \begin{array}{ccc} 1&0&0\\0&\cosh\xi& -\sinh\xi\\ 0&-\sinh\xi &\cosh\xi \end{array}\right ]\ ,
\end{equation}
and rotations about the $z$-axis
\begin{equation}
\label{eq2.3}
A_3(\theta)=\left [ \begin{array}{ccc} \cos\theta &-\sin\theta&0\\ \sin\theta&\cos\theta & 0\\ 0&0&1 \end{array}\right ] \ .
\end{equation}
The corresponding Lie algebra is spanned by
\begin{equation}
\label{eq2.4}
A_1'(0)=\left [ \begin{array}{ccc} 0&0&-1\\ 0&0&0\\ -1&0&0\end{array} \right ]\ ,\ A_2'(0)=\left [ \begin{array}{ccc} 0&0&0\\ 0&0&-1\\ 0&-1&0\end{array} \right ]\ ,\ A_3'(0)=\left [ \begin{array}{ccc} 0& -1&0\\1&0&0\\ 0&0&0 \end{array} \right ]\ .
\end{equation}

The boosts in the $x$-direction preserve $\partial_y$ and each of the two null planes in which it lies. Likewise, the boosts in the $y$-direction preserve $\partial_x$ and the two null planes in which it lies. Using this, one can see that a general $G$-transformation mapping any orthonormal basis of ${\mathbb M}^3$ to any other (preserving orientation and time orientation) can be constructed from a product $A_1(\zeta)A_2(\xi)A_3(\theta)$, whose parameters $\zeta$, $\xi$, and $\theta$ play the role of ``Euler angles''. We will sketch the argument. Consider two $\eta$-orthonormal basis (ONB) sets $\{ e_i \}$ and $\{{\tilde e}_i\}$, where $\{ e_i \}$ is the coordinate basis defined by the $x^i$ coordinates, with $e_3=\partial_z$ future-timelike and $\{ e_1, e_2\}$ right-handed (from here on, we simply say an \emph{oriented basis}). Likewise, ${\tilde e}_3$ will be future-timelike as well. The span of $\{ {\tilde e}_1,{\tilde e}_3\}$ is a timelike plane $\Pi$. It's a simple matter to find a (normalized spacelike) vector, call it $e_1'$, that lies in $\Pi$ and is orthogonal to $e_3$. It's also a simple exercise in linear algebra to find a rotation $A_3(\theta)$ about $e_3$ such that $A_3(\theta)e_1=e_1'$. This obviously leaves $e_3$ invariant, but maps $e_2$ to $e_2'=A_3(\theta)e_2$. Now apply a boost in the plane spanned by $e_2'$ and $e_3$. Such a boost $A_2(\xi)$ will leave $e_1'$ invariant, but we can choose $\xi$ such that $e_3':=A_2(\xi)e_3$, which remains timelike under a boost, lies in $\Pi$. This boost, incidentally, acts on $e_2'$ to produce $e_2'':=A_2(\xi)e_2'$. The plane $\Pi$ is now a coordinate plane for the ONB $\{ e_1',e_2'',e_3'\}$, with $\Pi= \Span \{ e_1', e_3'\}=\Span \{ {\tilde e}_1, {\tilde e}_3\}$, and $e_2''$ is normal (i.e., $\eta$-normal) to this plane. A final boost $A_1(\zeta)$ in the plane $\Pi$ preserves $e_2''$ and can be chosen such that ${\tilde e}_1=e_1'':=A_1(\zeta)e_1'$. Since $e_3'':=A_1(\zeta)e_3'$ must be orthogonal to $e_1''$, it follows that $\{ e_1'',e_2'',e_3''\}=\{ {\tilde e}_1, {\tilde e}_2, {\tilde e}_3\}$.

Then a curve of isometries may be written as
\begin{equation}
\label{eq2.5}
A(t)=A_1(\zeta(t))A_2(\xi(t))A_3(\theta(t))\ .
\end{equation}
If this curve passes through the identity isometry at $t=0$, we may write its tangent there as
\begin{equation}
\label{eq2.6}
\begin{split}
{\dot A}(0)=&\, {\dot A}_1(0){\dot \zeta}(0)+{\dot A}_2(0){\dot \xi}(0)+{\dot A}_3(0){\dot \theta}(0)\\
=&\, \left [ \begin{array}{ccc} 0&0&-1\\ 0&0&0\\ -1&0&0\end{array} \right ]{\dot \zeta}(0)+\left [ \begin{array}{ccc} 0&0&0\\ 0&0&-1\\ 0&-1&0\end{array} \right ]{\dot \xi}(0)+\left [ \begin{array}{ccc} 0& -1&0\\1&0&0\\ 0&0&0 \end{array} \right ]{\dot \theta}(0)\\
=&\, \left [ \begin{array}{ccc} 0& -{\dot \theta}(0)& -{\dot \xi}(0)\\ {\dot \theta}(0)&0&-{\dot \zeta}(0)\\ -{\dot \xi}(0)& -{\dot \zeta}(0)& 0 \end{array}\right ]\ .
\end{split}
\end{equation}
For any curve of isometries containing the identity at $t=0$, differentiation at $t=t_0$ can be accomplished using that $A(t_0+\epsilon)=A(t_0+\epsilon)A^{-1}(t_0)A(t_0)=:B(\epsilon)A(t_0)$ where $B(\epsilon):=A(t_0+\epsilon)A^{-1}(t_0)$ so $B(0)=\id$. But $B(\epsilon)$ can be written as a product $B(\epsilon)=A_1(\zeta(\epsilon))A_2(\xi(\epsilon))A_3(\theta(\epsilon))$, and so
\begin{equation}
\label{eq2.7}
{\dot A}(t_0)=\frac{d}{d\epsilon}\bigg \vert_{\epsilon=0}B(\epsilon)A(t_0) = {\dot B}(0) A(t_0)\ ,
\end{equation}
with ${\dot B}$ given by the right-hand side of \eqref{eq2.6}.

\subsection{Curves in the hyperboloid model} Given a hypersurface $S$ in a Riemannian manifold and vectors $U$ and $V$ tangent to it, a standard formula in Riemannian geometry gives
\begin{equation}
\label{eq2.8}
\nabla_U V = D_U V + K(U,V)\ ,
\end{equation}
where $\nabla$ is the connection on the ambient manifold, $D$ is the connection induced on the hypersurface, and $K$ is the vector-valued second fundamental form of the hypersurface. In the case of present interest, $\nabla$ is the Levi-Civita connection of the Minkowski metric $\eta=\langle\cdot , \cdot \rangle_{\eta}$. Since the vector $X$ from the origin to the unit hyperboloid is a unit future-timelike vector orthogonal to the hyperboloid, we may write a general vector $V$ as
\begin{equation}
\label{eq2.9}
V=P(V)-\langle X,V\rangle_{\eta} X\ ,
\end{equation}
where $P$ is orthogonal projection to the tangent space of the hyperboloid. The negative sign in the second term arises because $X$ is timelike. The second fundamental form can then be written as
\begin{equation}
\label{eq2.10}
K= P(\nabla X_{\flat})X\ ,
\end{equation}
where $X_{\flat}:=\langle \cdot , X\rangle_{\eta}$ is the 1-form metric-dual to $X$. If we write the first fundamental form of the hyperboloid as $h$ then the hyperboloid is umbilic in ${\mathbb M}^3$ such that $K=Xh$ (so that $\langle X, K(V,V)\rangle_{\eta} = -1$ for any unit vector $V$ tangent to the hyperboloid).

Now consider a (smooth) unit-speed curve $X(s)$ on the unit hyperboloid. The unit tangent vector is $T(s):=\frac{dX}{ds}$. It follows from the above formulas that
\begin{equation}
\label{eq2.11}
\nabla_T T = D_T T +K(T,T)=D_TT+X\ .
\end{equation}
Since $T$ is a unit vector, $\nabla_T T$ lies in its orthogonal complement, from which we see that $D_TT$ does as well. But $D_TT$ must be tangent to the hyperboloid and so orthogonal to $X$, so we write
\begin{equation}
\label{eq2.12}
D_TT = \kappa_g N\ ,
\end{equation}
where $N$ is called the \emph{principal normal vector} to the curve $s\mapsto X(s)$ and $\kappa_g$ is the \emph{geodesic curvature}. The sign choices are such that $\{ T,N,X\}$ is an orthonormal oriented basis oriented such that $X$ is future-pointing and the vector $\kappa_g N$ points to the concave side of the curve, viewed as a curve in the hyperboloid, whenever $\kappa_g$ is not zero. Now denote $\eta(\nabla_T T,\nabla_T T)=:\epsilon\kappa^2$ where $\epsilon=1$ if $\nabla_T T$ is spacelike, $0$ if it's null, and $-1$ if it's timelike. Then from \eqref{eq2.11} we can relate the curvature $\kappa$ of $X$ as a space curve in $({\mathbb M}^3,\eta)$ to its geodesic curvature $\kappa_g$ in the hyperboloid by
\begin{equation}
\label{eq2.13}
\kappa_g^2=\epsilon\kappa^2+1\ .
\end{equation}

Note that $\langle T, \nabla_T N\rangle_{\eta} = - \langle N,\nabla_T T\rangle_{\eta} = -\langle N,D_T T\rangle_{\eta} = -\kappa_g$. By similar reasoning, we see that $\langle X, \nabla_T N\rangle_{\eta} =0$, and of course since $N$ is a unit vector then $\langle N, \nabla_T N\rangle_{\eta} =0$. Thus $\nabla_T N = -\kappa_g T$. Collecting our results, we have that the $\{ T,N,X\}$ basis evolves along the curve according to the \emph{Frenet-Serret equations} in ${\mathbb H}^2$, which are
\begin{equation}
\label{eq2.14}
\begin{split}
\nabla_T T=&\, \kappa_g N +X\ ,\\
\nabla_T N=&\, -\kappa_g T\ ,\\
\nabla_T X=&\, T\ .
\end{split}
\end{equation}

\subsection{Curves and axes}
Here we follow closely the work of \cite{SdRT} for the ${\mathbb S}^2$ case, making changes as necessary. Choose an arbitrary vector ${\tilde v}\in {\mathbb M}^3\backslash \{0\}$, which can be either timelike, spacelike, or null. It is convenient to define ${\tilde v}=av$ where $a=\sqrt{|\eta({\tilde v},{\tilde v})|}$ when $v$ is not null. For presentation purposes, we will introduce the scale factor even when ${\tilde v}$ is null, but in that case $a\neq 0$ will not be determined and $v$, being null, cannot be normailzed. Then $\eta({\tilde v}, {\tilde v})=\epsilon = 0,\pm 1$ depending on whether $v$ is null ($0$), spacelike ($+1$), or timelike ($-1$).

Fix a unit speed curve $X:I\to {\mathbb M}^3$ on the unit hyperboloid, and define the functions
\begin{equation}
\label{eq2.15}
\begin{split}
\tau(s):=&\, \langle T(s),v\rangle_{\eta}\ , \\
\nu(s):=&\, \langle N(s),v\rangle_{\eta}\ , \\
\mu(s):=&\, \langle X(s), v\rangle_{\eta}\ .
\end{split}
\end{equation}
Then we can write
\begin{equation}
\label{eq2.16}
v= \tau(s) T(s) +\nu(s) N(s) -\mu(s)X(s)
\end{equation}
and
\begin{equation}
\label{eq2.17}
\epsilon=\eta(v,v)=\tau^2(s)+\nu^2(s)-\mu^2(s)\ .
\end{equation}

The next result shows that if the geodesic curvature of the curve $X(s)$ takes the form $\kappa_g(s)=a\tau(s)$, a choice corresponding to a flow by isometries under CSF as we will see in the next subsection, then $\tau$, $\nu$, and $\mu$ obey a certain autonomous system of differential equations along $X(s)$. The analogous result for curves in ${\mathbb S}^2$ is found in \cite[Proposition 3.1]{SdRT}.

\begin{proposition}\label{proposition2.1}
Along any smooth unit speed curve $X(s)$ on the hyperboloid we have that $\kappa_g(s)=a\tau(s)$ if and only if
\begin{equation}
\label{eq2.18}
\begin{cases} \tau'(s)=a\tau(s)\nu(s) +\mu(s),\\ \nu'(s)=-a\tau^2(s),\\ \mu'(s)= \tau(s).\end{cases}
\end{equation}

\end{proposition}

\begin{proof}
Equations \eqref{eq2.14} imply that
\begin{equation}
\label{eq2.19}
\begin{split}
\tau'(s)=&\, \kappa_g(s)\nu(s)+\mu(s)\ ,\\
\nu'(s)=&\, -\kappa_g(s)\tau(s)\ , \\
\mu'(s)=&\, \tau(s)\ .
\end{split}
\end{equation}

To prove the forward implication, simply substitute $\kappa_g(s) = a\tau(s)$ into \eqref{eq2.19}.

To prove the reverse implication, note that we can combine equations \eqref{eq2.18} and \eqref{eq2.19} to write that $\left (a\tau-\kappa_g\right )\nu=0$ and $\left (a\tau-\kappa_g\right )\tau=0$. But if $\tau$ and $\nu$ both vanish, then $v$ must be orthogonal to both $T$ and $N$ and hence parallel to $X$, so $v$ is normal to the hyperboloid. But $v$ is a constant vector (i.e., it is parallel in $({\mathbb M}^3,\eta)$), so this can only happen at isolated points along a nontrivial curve $X$, so it must instead be that $a\tau-\kappa_g=0$.
\end{proof}

Next we show that solutions of the system \eqref{eq2.18} are always realized by actual curves.

\begin{proposition}\label{proposition2.2}
Given a solution $(\tau(s),\nu(s),\mu(s))$ of \eqref{eq2.18} obeying initial conditions $(\tau_0,\nu_0,\mu_0)=(\tau(0),\nu(0),\mu(0))$,
there is a smooth unit speed curve $X:I\to {\mathbb M}^3$ on the unit hyperboloid satisfying \eqref{eq2.15}.
\end{proposition}

\begin{proof} Choose a point $X_0$ on the hyperboloid and two vectors $T_0$ and $N_0$ such that $\{ T_0, N_0, X_0\}$ is an oriented orthonormal frame. Then define $v$, normalized as above (if non-null), by $v=\tau_0T_0 +\nu_0N_0 -\mu_0X_0$. Now, given $\tau(s)$ for $s\in I$ and given $a>0$, define a function $k:I\to {\mathbb R}:s\mapsto a\tau(s)$. By the fundamental theorem for curves in ${\mathbb H}^2$, there is a unique curve $X:I\to {\mathbb M}^3$ such that $X(0)=X_0$, $X'(0)=:T(0)=T_0$, and $N(0)=N_0$, whose curvature is $k(s)$. This curve must satisfy equations \eqref{eq2.14} with $\kappa_g(s)=k(s)$. Then equations \eqref{eq2.15} hold for any constant vector $v$, and so hold for $v=\tau_0T_0+\nu_0N_0-\mu_0X_0$.
\end{proof}

\subsection{CSF and flow by isometries}
Having defined a moving basis along an arbitrary curve in the hyperboloid, we can now define the \emph{curve shortening flow} in the hyperboloid to be a flow $X_t(s)=X(t,s)$ of smooth curves $X_t(s)=\left ( x_t(s),y_t(s),z_t(s)\right )$ with principal normal $N_t$, for $t\in J \subset {\mathbb R}$ some connected interval about $t=0$, such that
\begin{equation}
\label{eq2.20}
\begin{split}
\left \langle \frac{\partial X_t}{\partial t}, N_t \right \rangle_{\eta} = &\, \kappa_g\ ,\\
\left \langle \frac{\partial X_t}{\partial t}, X_t \right \rangle_{\eta} = &\, 0\ ,\\
\left \vert X_0 \right \vert_{\eta}^2=x_0^2+y_0^2-z_0^2= &\, -1\ ,\ z_0\ge 1\ .
\end{split}
\end{equation}
The last two equations are equivalent to the conditions $\left \langle X_t,X_t \right \rangle_{\eta} = -1$ and $X_0=\left ( x_0,y_0,z_0\right )$.

We are interested in those solutions of \eqref{eq2.20} that are of the form
\begin{equation}
\label{eq2.21}
X_t=A(t)X_0\ , \ t\in J\ ,\ A(0)=\id\ ,
\end{equation}
where $A(t):=A_1(\zeta(t))A_2(\xi(t))A_3(\theta(t))$ with $\zeta(0)=\xi(0)=\theta(0)=0$, with the $A_i$ defined in equations \eqref{eq2.1}--\eqref{eq2.3}. Then $t\mapsto X_t$ is called a \emph{flow by isometries}. Motion by an isometry preserves the curvature, so $\kappa_g(s,t)=\kappa)g(s)$ (i.e., it is independent of $t$).

As we saw in Proposition \ref{proposition2.1}, if there is a (normalized or null) vector $v$ such that $a\langle T,v\rangle_{\eta} =\kappa_g$, then a system of three scalar equations governs the moving frame. The next result shows that if a curve shortening flow is also a flow by isometries, there is always such a $v$ which is constant with respect to the flow parameter $t$. For convenience, we work with the unrescaled vector ${\tilde v}=av$ ($a(s)$ is used for another purpose in the proof). We follow the proof of the analogous result for curves in ${\mathbb S}^2$ \cite[Theorem 2.2]{SdRT}.

\begin{proposition}\label{proposition2.3}
Let $X(s,t)=:X_t(s)$ be a smooth function of two variables such that $X_t(s)$ is a regular unit speed curve for each $t\in J$, lying in the hyperboloid $z=\sqrt{1+x^2+y^2}$ in ${\mathbb M}^3$ and evolving by isometries as in \eqref{eq2.21}. Then \eqref{eq2.20} holds if and only if there is a ${\tilde v}\in {\mathbb M}^3\backslash \{0\}$ such that
\begin{equation}
\label{eq2.22}
\left \langle T(s),{\tilde v}\right \rangle_{\eta} =\kappa_g(s)
\end{equation}
for all $s\in I$, where $T(s)=\frac{dX_{t_0}}{ds}=:X_{t_0}'(s)$ and $\kappa_g(s)$ is the geodesic curvature of $X_{t_0}(s)=:X(s)$.
\end{proposition}

\begin{proof} Let $X_t(s)=A(t)X_0$, so that $X$ flows by isometries. Then at any $t=t_0\in J$, we compute
\begin{equation}
\label{eq2.23}
\frac{\partial}{\partial t}\bigg \vert_{t_0} X_t(s) = \left [ \begin{array}{ccc} 0& -{\dot \theta}& -{\dot \xi}\\ {\dot \theta}&0&-{\dot \zeta}\\ -{\dot \xi}& -{\dot \zeta}& 0 \end{array}\right ]\left [ \begin{array}{c} x_{t_0}(s)\\y_{t_0}(s)\\z_{t_0}(s) \end{array} \right ] = \left [ \begin{array}{c} -{\dot \theta}y_{t_0} -{\dot \xi}z_{t_0}\\ {\dot \theta}x_{t_0}-{\dot \zeta}z_{t_0}\\ -{\dot \xi}x_{t_0}-{\dot \zeta}y_{t_0} \end{array} \right ] \ ,
\end{equation}
using \eqref{eq2.7} and \eqref{eq2.6}. Let $T_{t_0}(s):=X_t'(s)\big \vert_{t=t_0}$ be the unit tangent vector to $X_{t_0}(s)$ and let $N_{t_0}(s)$ be the unit normal vector to $X_{t_0}(s)$ in the tangent space to the hyperboloid, defined so that $\{T_{t_0} ,N_{t_0},-X_{t_0}\}$ is an oriented orthonormal basis. Writing $N_{t_0}(s)=(a(s),b(s),c(s))$, we have
\begin{equation}
\label{eq2.24}
\begin{split}
\left \langle N_{t_0}, X_{t_0}\right \rangle_{\eta} =&\, a(s)x_{t_0}(s)+b(s)y_{t_0}(s)-c(s)z_{t_0}(s)=0\, \\
\left \langle N_{t_0}, T_{t_0}\right \rangle_{\eta} =&\, a(s)x_{t_0}'(s)+b(s)y_{t_0}'(s)-c(s)z_{t_0}'(s)=0\ ,
\end{split}
\end{equation}
so that at each $s$ along $X_{t_0}$ we have
\begin{equation}
\label{eq2.25}
N_{t_0}=(a(s),b(s),c(s))=\left ( y_{t_0}'z_{t_0}-y_{t_0}z_{t_0}',z_{t_0}'x_{t_0}-z_{t_0}x_{t_0}',-x_{t_0}'y_{t_0}+x_{t_0}y_{t_0}'\right )\ .
\end{equation}
It's not difficult to check that this vector has norm $1$ and has the correct sign. Then
\begin{equation}
\label{eq2.26}
\begin{split}
\left \langle N_{t_0}, \frac{\partial X_t}{\partial t}\bigg\vert_{t=t_0}\right \rangle_{\eta}=&\, \left [ \begin{array}{ccc}a(s)& b(s)& c(s) \end{array}\right ]\left [ \begin{array}{ccc}1& 0& 0\\0&1&0\\ 0&0&-1 \end{array}\right ]\left [ \begin{array}{c} -{\dot \theta}y -{\dot \xi}z\\ {\dot \theta}x-{\dot \zeta}z\\ -{\dot \xi}x-{\dot \zeta}y \end{array} \right ]\\
=&\, {\dot \zeta} \left ( xyy'-xzz'-x'y^2+x'z^2\right )+{\dot \xi} \left ( -xx'y+zz'y+x^2y'-z^2y'\right )\\
&\, +{\dot \theta}\left ( -xx'z-yy'z +x^2z'+y^2z'\right )
\end{split}
\end{equation}
where we've removed the $t_0$ subscripts on the right to lessen the clutter. Using that $x^2+y^2-z^2=-1$ and, therefore, that $xx'+yy'=zz'$, this last line simplifies and implies that, for any curve moving by isometries described by $A(\zeta(t),\xi(t),\theta(t))$, we have
\begin{equation}
\label{eq2.27}
\left \langle N_{t_0}, \frac{\partial X_t}{\partial t}\bigg\vert_{t=t_0}\right \rangle_{\eta} ={\dot \zeta}x'-{\dot \xi}y'-{\dot \theta}z'=\langle T_{t_0},{\tilde v}\rangle_{\eta}\ ,
\end{equation}
where ${\tilde v}:=\left ( {\dot \zeta}, -{\dot \xi},{\dot \theta}\right )$. (Note that ${\tilde v}$ can be either timelike, spacelike, or null.)

To prove necessity (``only if''), by \eqref{eq2.20} we have $\left \langle \frac{\partial X_{t_0}}{\partial t}, N_{t_0} \right \rangle_{\eta} = \kappa_{g}(s)$, so
\begin{equation}
\label{eq2.28}
\left \langle T_{t_0}(s),{\tilde v}\right \rangle_{\eta} \equiv \left \langle T(s),{\tilde v}\right \rangle_{\eta}=\kappa_g(s)\ .
\end{equation}

To prove sufficiency, we note that one can by direct computation prove that if $A_i(t)$ is any of the matrices \eqref{eq2.1}--\eqref{eq2.3} then in matrix notation $(A_i'(t))^T \eta A_i(t)=(A_i'(0))^T\eta A_0(0)$ where $M^T$ denotes the transpose of $M$ (note that $A_i(0)=\id$). For a flow of the form $X(s,t)=A_i(t)X(s)$, we have $N(s,t)=A_i(t)N(s)$, so
\begin{equation}
\label{eq2.29}
\begin{split}
\kappa_g(s,t)=&\, \left \langle \frac{\partial X}{\partial t}(s,t), N(s,t)\right \rangle= X^T(s) \left (A_i'(t)\right )^T\eta A_i(t) N(s) =X^T(s)(A_i'(0))^T\eta A_0(0) N(s)\\
=&\,\kappa_g(s,0)\ .
\end{split}
\end{equation}
Thus, any of the families of isometries listed in \eqref{eq2.1}--\eqref{eq2.3} produce curve shortening flows such that $\kappa_g(s)$ is independent of $t$. So if there is a vector $v$ and a unit speed curve $X(s)$ with curvature function $\kappa_g(s)$ such that $\langle X'(s),v\rangle=\kappa_g(s)$, then a flow by any of these families will be a curve shortening flow with $\kappa_g(s,t)=\kappa_g(s,0)\equiv\kappa_g(s)$.
\end{proof}

\begin{remark}
It is always possible to reduce \eqref{eq2.29} to the form $\frac{\partial X_t}{\partial t}=\kappa_gN_t$ by a reparametrization $u\mapsto\varphi(u,t)$ of the curve $X(u,t)$ (e.g., \cite[Proposition 1.1]{CZ}).
\end{remark}

\section{Examples}
\setcounter{equation}{0}


\subsection{Geodesics} These are solutions of \eqref{eq2.18} with $\tau\equiv 0$, so $\mu\equiv 0$ as well, and $\nu$ is constant. They are fixed points of \eqref{eq3.1}. Since $\mu=0$, geodesics are plane curves lying in the plane orthogonal to $v$.

We may view the system \eqref{eq2.18} as an autonomous third-order system
\begin{equation}
\label{eq3.1}
\Phi'= F_a(\Phi)\ ,
\end{equation}
where $\Phi=(\tau,\nu,\mu)$, $a>0$ is a constant, and $F_a:{\mathbb R}^3\to{\mathbb R}^3:(\tau,\nu,\mu) \mapsto (a\tau\nu+\mu,-a\tau^2,\tau)$. We now show that at fixed points, the differential of $F_a$ has real eigenvalues of all three signs (positive, negative, and zero), so geodesics are unstable fixed points but attract in one direction.

\begin{lemma}\label{lemma3.1}
For $F_a$ as above, then $F_a=0 \Leftrightarrow (\tau, \nu, \mu)=(0,\pm 1,0)=\pm e_2$, and the corresponding curves are geodesics. The eigenvalues of $dF_a\vert_{e_2}$ are
\begin{equation}
\label{eq3.2}
\lambda=\ 0,\ \frac{a+\sqrt{a^2+4}}{2}>0,\ \frac{a-\sqrt{a^2+4}}{2}<0\ .
\end{equation}
The eigenvalues of $dF_a\vert_{-e_2}$ are the negatives of the above eigenvalues.
\end{lemma}

Specifically, when $a=1$ and $(\tau, \nu, \mu)=(0,1,0)$ (respectively, $(\tau, \nu, \mu)=(0,1,0)$),  then $\lambda=0,\frac{1+\sqrt{5}}{2}, \frac{1-\sqrt{5}}{2}$ (respectively, $\lambda=0,\frac{-1+\sqrt{5}}{2}, \frac{-1-\sqrt{5}}{2}$).

\begin{proof}The fixed points at $\pm e_2$ are obvious. Then since $\mu=0$ and $\tau=0$, we have from \eqref{eq2.15} that $v$ is normal to a timelike plane containing $X$ and $T$. The curve $X(s)$ is then the intersection curve of this timelike plane, which contains the origin, and the hyperboloid. It is therefore a hyperbolic great circle, and thus a geodesic (this can also be seen from \eqref{eq2.22}).

The differential of $F_a$ is given by
\begin{equation}
\label{eq3.3}
dF_a=\begin{pmatrix}a\nu &a\tau&1  \\ -2a\tau & 0&0\\ 1&0&0 \end{pmatrix}\ .
\end{equation}
Hence, at $\pm e_2=(0,\pm 1,0)$ we have
\begin{equation}
\label{eq3.4}
dF_a\vert_{\pm e_2}=\begin{pmatrix} \pm a &0&1  \\ 0 & 0&0\\ 1&0&0 \end{pmatrix}\ .
\end{equation}
Therefore, $\lambda$ is an eigenvalue of $dF_a\vert_{\pm e_2}$ iff $\lambda^2(\pm a-\lambda)+\lambda=-\lambda\left ( \lambda^2 \mp a\lambda -1\right )=0$.
\end{proof}

\subsection{Horocycles}
In the Poincar\'e disk model, horocycles are Euclidean circles tangent to the boundary of the Poincar\'e disk at one point and otherwise lying within the disk. If the point of tangency is $(1,0)$ then the horocycle with centre $(\varpi,0)$ can be written in parametrized form as
\begin{equation}
\label{eq3.5}
u(\rho)= \varpi +(1-\varpi)\cos \frac{\rho}{(1-\varpi)}\quad  , \quad
v(\rho)= (1-\varpi)\sin \frac{\rho}{(1-\varpi)}\ ,
\end{equation}
where $\rho$ is a parameter. By inverting the stereographic projection $(u,v)=\left ( \frac{x}{1+z},\frac{y}{1+z}\right )$ where $z=\sqrt{1+x^2+y^2}$, we obtain a parametric description of this horocycle in the hyperboloidal model
\begin{equation}
\label{eq3.6}
x=\frac{\varpi^2 s^2 +2\varpi -1}{2\varpi (1-\varpi)}\quad  , \quad
y=s\quad  , \quad
z=\frac{\varpi^2 s^2 +1}{2\varpi (1-\varpi)}-1\ .
\end{equation}
Here $s$ is an arclength parameter. It is straightforward to compute the Frenet frame
\begin{equation}
\label{eq3.7}
\begin{split}
T=&\, \left ( \frac{\varpi s}{(1-\varpi)},1,
\frac{\varpi s}{(1-\varpi)}\right )\quad \, \quad
N= \left ( \frac{\varpi^2 s^2-2\varpi^2+2\varpi-1}{2\varpi (1-\varpi)}, s, \frac{\varpi^2 s^2 -2\varpi +1}{2\varpi (1-\varpi)}\right )\ ,\\
X=&\, \left ( \frac{\varpi^2 s^2 +2\varpi -1}{2\varpi (1-\varpi)}, s, \frac{\varpi^2 s^2 +1}{2\varpi (1-\varpi)}-1\right )\ .
\end{split}
\end{equation}
Choosing $v=(0,1,0)$ we obtain
\begin{equation}
\label{eq3.8}
\tau=1\ ,\ \nu=-\mu=-s\ .
\end{equation}
Entering these values into the system \eqref{eq2.18}, we see that the system is satisfied if $a=1$. Then the relation $\kappa_g=a\tau$ yields $\kappa_g=1$, as it must for horocycles.

As discussed by Grayson \cite{Grayson}, under the curve shortening flow horocycles remain horocycles, and so evolve self-similarly, with curvature $\kappa_g=1$ throughout the flow. In the next section, we will seek all solutions of the system \eqref{eq2.18} corresponding to self-similar evolutions with $v=(0,1,0)$. This will be the case of curves which evolve purely by a boost in ${\mathbb M}^3$.

\subsection{Hypercycles} In the Poincar\'e disk model, hypercycles are arcs of Euclidean circles that end when they meet the boundary of the Poincar\'e disk transversally. Under CSF they do not evolve by isometries, but do evolve under the composition of isometries and rescalings.

Applying a rotation if necessary, we can place the centre of the Euclidean circle along one of the axes, say at $(\varpi,0)$ where $\varpi>0$. Let the Euclidean radius of the circle be $c<1+\varpi$. We can parametrize the hypercycle as we would any such Euclidean circle, say as
\begin{equation}
\label{eq3.9}
u(t)=\varpi + c\cos\frac{t}{c}\quad ,\quad
v(t)= c\sin\frac{t}{c}\ ,
\end{equation}
with the domain of $t$ chosen so that $u^2+v^2< 1$.

In the hyperboloidal model,
by using stereographic projection to lft the above parametrized curve to the unit hyperboloid we obtain
\begin{equation}
\label{eq3.12}
\begin{split}
x(t)=&\, \frac{2\left ( \varpi + c\cos\frac{t}{c}\right )}{1-c^2-\varpi^2 -2c\varpi\cos\frac{t}{c}}\quad , \quad
y(t)= \frac{2c\sin\frac{t}{c}}{1-c^2-\varpi^2 -2c\varpi\cos\frac{t}{c}}\ ,\\
z(t)=&\, \frac{1+c^2+\varpi^2 +2c\varpi\cos\frac{t}{c}}{1-c^2-\varpi^2 -2c\varpi\cos\frac{t}{c}}\ .
\end{split}
\end{equation}
These curves are curves of intersection of the hyperboloid $z=\sqrt{1+x^2+y^2}$ in ${\mathbb M}^3$ with the planes
\begin{equation}
\label{eq3.13}
2\varpi \left ( x+\frac{1}{\varpi}\right ) +\left (c^2-\varpi^2-1\right ) (z+1)=0\ ,
\end{equation}
for $c+\varpi>1$. These planes are timelike.

In his seminal paper, Grayson computes the evolution of hypercycles in ${\mathbb H}^2$ under CSF \cite[p 76]{Grayson}. The curve remains a hypercycle, but the curvature evolves as
\begin{equation}
\label{eq3.14}
\kappa_g(s,t) =\kappa_g(t)= \frac{1}{\sqrt{1-Ae^{2t}}}\ .
\end{equation}

\section{Curve shortening flows evolving by isometries}
\setcounter{equation}{0}

\subsection{Properties of solutions} From the system of equations \eqref{eq2.18}, we can deduce the following lemmata. We exclude the case of $\mu$ identically zero. In this case, the curve $X(s)$ and the constant vector $v$ are orthogonal, and then $X$ is a geodesic. We also take $X(s)$ to be defined for all $s\in {\mathbb R}$. Note that by equations \eqref{eq2.15} and \eqref{eq2.22} $tau$, $\nu$, $\mu$, and $\kappa$ will be bounded on the intersection of $X$ with any compact subset of ${\mathbb H}^2$, so inextendible curves trapped in a bounded set have infinite arclength. For those that escape any bounded set, the distance from any chosen $p\in{\mathbb H}^2$ increases without bound, so the arclength parameter is unbounded in this case as well.

\begin{lemma}[Behaviour of $\mu$]\label{lemma4.1}
Let $(\tau(s),\nu(s),\mu(s))$ be a solution of the system \eqref{eq2.18} and let $s_0$ be a critical point of $\mu$. If $\mu$ is not identically zero then $s_0$ is the unique critical point of $\mu$. If $\mu(s_0)>0$ it is a global minimum. If $\mu(s_0)<0$ it is a global maximum. Furthermore, $\mu$ has at most one zero. If $\mu$ converges as $s\to\infty$, then in fact $\lim_{s\to \infty}\mu(s)=0$. Likewise, if $\mu$ converges as $s\to-\infty$, then $\lim_{s\to -\infty}\mu(s)=0$. Otherwise, $\mu$ diverges to $\pm\infty$.

\end{lemma}

\begin{proof}
We have from \eqref{eq2.18} that
\begin{equation}
\label{eq4.1}
\mu''(s)=\tau'(s)=a\tau(s)\nu(s)+\mu(s) = a\mu'(s)\nu(s)+\mu(s)\ .
\end{equation}
Now at a critical point we have $\mu'(s_0)=\tau(s_0)=0$. Since by \eqref{eq2.17} $\nu$ cannot diverge at $s_0$ then from \eqref{eq4.1} we see that $\mu''(s_0)=\mu(s_0)$. If $\mu(s_0)>0$ then the second derivative test implies that it must be a local minimum. Now say there is another critical point at $s_1$, and that no critical point lies between $s_0$ and $s_1$. Then necessarily $\mu(s_1)>\mu(s_0)>0$, so $\mu(s_1)$ is also a local minimum and so there must be a local maximum between these critical points, which is a contradiction. Hence $s_0$ is the unique critical point, and therefore a global minimum. The dual statement for the case of $\mu(s_0)<0$ now follows in precisely similar fashion. In either of these cases, $\mu $ has no zero.

If $\mu(s_0)=0$ but $\mu(s)$ is not identically zero, then either (i) $\mu(s)<0$ for $s_0<s<s_0+\delta$ (for some $\delta >0$) or (ii) $\mu(s)>0$ for $s_0<s<s_0+\delta$. In case (i), then by the above argument the next critical point at some $s_1>s_0$ must have $\mu''(s_1)<0$ and would therefore be a local maximum, a contradiction; in case (ii) it must have $\mu''(s_1)>0$, again a contradiction. So there is no critical point at $s_1>s_0$. A similar argument shows that no critical point at $s_1<s_0$ either. Therefore, if $\mu(s_0)=0$ then either $s_0$ is the unique critical point, or $\mu$ has no critical points at all in this case, and in either situation $s_0$ is then the unique zero of $\mu$.

Since $\mu$ has at most one critical point $s_0$, on the domain $s<s_0$ we may take $\mu$ to be monotonic and $\tau$ to have a sign. The same statements are true on the domain $s>s_0$, and for all $s$ if $\mu$ has no critical point. Then if $|\mu|$ does not diverge to $\infty$, $\mu$ will converge. By way of contradiction, assume that $\mu$ converges to a constant $c\neq 0$. Then from the third equation in \eqref{eq2.18} and recalling the $\tau$ has a sign, we have $\tau\to 0$. But then $\eta(v,v)=\tau^2+\nu^2-\mu^2$, so $\tau^2+\nu^2$ is bounded, which implies that $\nu^2$ cannot diverge to $\infty$. Then the first equation in \eqref{eq2.18} yields $\tau'\to c$. But then $\tau$ would diverge, a contradiction, unless $c= 0$.
\end{proof}

\begin{lemma}[Behaviour of $\nu$]\label{lemma4.2}
The value $s_0$ is a critical point of $\nu$ if and only if $s_0$ is a critical point of $\mu$, and then $s_0$ is a point of inflection for $\nu$, which is therefore monotonic. Either $\nu$ is bounded and thus converges or it diverges to $+\infty$ as $s\to -\infty$ or to $-\infty$ as $s\to +\infty$.
\end{lemma}

\begin{proof}
Since by \eqref{eq2.18} we have $\nu'=-a\tau^2=-a(\mu')^2$, then it is obvious that $s_0$ is a critical point of $\nu$ if and only if $s_0$ is a critical point of $\mu$, and that $\nu(s)$ is monotonic so any critical point is an inflection point. It also follows from monotonicity that if $\nu$ is bounded it converges as $s\to\pm\infty$, and since it decreases monotonically, if it is not bounded it diverges as claimed.
\end{proof}

Notice that we have both a monotone quantity $\nu$ and a conserved quantity $\eta(v,v)=\tau^2+\nu^2-\mu^2$.

\begin{lemma}[Linear growth]\label{lemma4.3}
Let $(\tau(s),\nu(s),\mu(s))$ be a solution of the system \eqref{eq2.18} such that $\mu$ diverges as $s\to\infty$. Then $\tau(s),\nu(s),\mu(s)\in {\mathcal O}(s)$; i.e., they are each bounded above in magnitude by $Cs$ for some constant $C>0$. This is also true as $s\to -\infty$.
\end{lemma}

\begin{proof}
Since $x\mapsto x^2$ is a convex function, Jensen's inequality yields for $s>s_0$ that
\begin{equation}
\label{eq4.2}
\left ( \frac{1}{(s-s_0)}\int\limits_{s_0}^s \tau(r) dr \right )^2 \le \frac{1}{(s-s_0)}\int\limits_{s_0}^s\tau^2(r) dr\ .
\end{equation}
Since $\mu'=\tau$ and $\nu'=-a\tau^2$, the above inequality implies that
\begin{equation}
\label{eq4.3}
\begin{split}
&\, \frac{\left ( \mu(s)-\mu(s_0)\right )^2}{\left ( s-s_0\right )^2} \le -\frac{\left ( \nu(s)-\nu(s_0)\right )}{a\left ( s-s_0\right )}\\
\implies &\, \mu(s) \le \mu(s_0)+\sqrt{\frac{1}{a}\left ( s-s_0\right ) \left ( \nu(s_0)-\nu(s)\right )}\ .
\end{split}
\end{equation}
In the last line, we recall that $\nu$ is a decreasing function. Writing $\mu_0:=\mu(s_0)$ and $\nu_0:=\nu(s_0)$, then
\begin{equation}
\label{eq4.4}
\eta(v,v)
\ge \tau^2+\nu\left [\nu+\frac{(s-s_0)}{a}\right ] -2\mu_0 \sqrt{\frac{1}{a}\left ( s-s_0\right ) \left ( \nu_0-\nu\right )} -\frac{(s-s_0)}{a} \nu_0-\mu_0^2\ ,
\end{equation}
so that
\begin{equation}
\label{eq4.5}
\frac{\eta(v,v)+\mu_0^2}{(s-s_0)^2}+\frac{\nu_0}{a(s-s_0)} +2\mu_0 \sqrt{\frac{\left ( \nu_0-\nu\right )}{a(s-s_0)^3}}\ge \frac{\tau^2}{(s-s_0)^2}+\frac{\nu}{(s-s_0)}\left [\frac{\nu}{(s-s_0)}+\frac{1}{a}\right ]\ .
\end{equation}
If there is a sequence of values $s_i\to\pm\infty$ such that either $|\tau(s_i)|$ or $|\nu(s_i)|$ (or both) grows faster than linearly, then this inequality cannot hold. But then by \eqref{eq4.3}, $|\mu|$ also cannot grow faster than linearly.
\end{proof}

\begin{lemma}[Bounded curvature]\label{lemma4.4}
The curvature and $\tau$ are bounded.
\end{lemma}

\begin{proof} Since $\kappa_g=a\tau$, then $\kappa_g$ is bounded iff $\tau$ is.

Consider first $s\to\infty$. If $\nu$ diverges, then necessarily $\nu\to -\infty$ as $s\to\infty$, so take $s_1$ such that $a\nu(s_1)+2<0$. Since $\nu$ is monotonic, then $a\nu(s)+2<0$ for all $s\ge s_1.$ Since $\mu$ has at most one critical point and therefore $\tau$ has at most one zero, by increasing $s_1$ if necessary we can take $\tau(s)>0$ and $\mu(s)>0$ for $s\ge s_1$. (We will deal with other cases at the end.)

Under these circumstances, we have from equations \eqref{eq2.18} and \eqref{eq2.17} that
\begin{equation}
\label{eq4.6}
\begin{split}
\tau'=&\, a\nu\tau +\mu=av\tau +\sqrt{\tau^2+\nu^2-\eta(v,v)}
\le  a\nu\tau +\sqrt{\tau^2+\nu^2+\left \vert \eta(v,v)\right \vert }\\
\le &\, a\nu\tau +2\sqrt{\tau^2+\nu^2}\ ,
\end{split}
\end{equation}
where we may have to increase $s_1$ so that the last inequality holds (indeed, we have $\tau'\le a\nu\tau +\sqrt{\tau^2+\nu^2}$ if $v$ is null or spacelike). Furthermore, since $\nu<0$ and $\tau>0$, then $\sqrt{\tau^2+\nu^2}<\sqrt{\tau^2+\nu^2-2\tau\nu}=\tau-\nu$, so we get from \eqref{eq4.6} that
\begin{equation}
\label{eq4.7}
\tau'<a\nu\tau+2\tau-2\nu=(a\nu+2)\tau-2\nu\ .
\end{equation}
But $a\nu+2<0$, and so whenever $\tau>\frac{2\nu}{a\nu+2}$ equation \eqref{eq4.7} will yield $\tau'<2\nu-2\nu=0$.
Rewriting
\begin{equation}
\label{eq4.8}
\frac{2\nu}{a\nu+2} =\frac{2}{a} -\frac{4}{a(a\nu(s)+2)}\le \frac{2}{a} -\frac{4}{a(a\nu(s_1)+2)}=\frac{2\nu(s_1)}{a\nu(s_1)+2}
\end{equation}
using the monotonicity of $\nu$, we therefore obtain for $s\ge s_1$ that
\begin{equation}
\label{eq4.9}
0<\tau(s)\le \max \left \{ \tau(s_1),\frac{2\nu(s_1)}{a\nu(s_1)+2}\right \}\ .
\end{equation}

Next we deal briefly with the case where $\tau(s_1)<0$ and $\mu(s_1)<0$. Then subsequently $\tau(s)<0$ and $\mu(s)<0$ for all $s>s_1$. As above, $\nu(s)<0$ and $\nu\to -\infty$. Repeating the calculations of equations \eqref{eq4.6} and \eqref{eq4.7} with the appropriate sign changes, we now have
\begin{equation}
\label{eq4.10}
\begin{split}
\tau'=&\, a\nu\tau-\sqrt{\tau^2+\nu^2-\eta(v,v)}
\ge a\nu\tau-\sqrt{\tau^2+\nu^2+\left \vert \eta(v,v)\right \vert }\\
\ge &\, a\nu\tau-2\sqrt{\tau^2+\nu^2 }
\ge a\nu\tau-2(\tau+\nu)\\
\ge &\, (a\nu-2)\tau\ ,
\end{split}
\end{equation}
where in the last line we dropped the term $-2\nu$ since $\nu(s)<0$. Using $a\nu(s)-2<0$ for $s\ge s_1$ when $s_1$ large enough, then we have that $\tau'>0$ whenever $\tau<\frac{1}{a\nu-2}<0$. Alternatively, we can simply observe that the logarithmic derivative of $\tau$ is negative, so the magnitude of $\tau$ decreases. Hence $\tau$ is bounded.

To treat the limit $s\to -\infty$, note that we can replace $s$ by $r:=-s$ in the derivatives in \eqref{eq2.18} and then treat the limit $r\to\infty$. If we also replace $\mu\mapsto -\mu$ and $\nu\mapsto -\nu$, we recover the same system as \eqref{eq2.18}. We will again obtain that $\tau$ is bounded as $r\to\infty$ and therefore as $s\to -\infty$.
\end{proof}

\begin{lemma}[Convergence]\label{lemma4.5}
If either $\mu$ or $\nu$ is bounded, then all three functions $\tau$, $\nu$, and $\mu$ converge, with $\mu\to 0$, $\tau\to 0$, and $\nu\to\pm\sqrt{\eta(v,v)}$. This can occur only when $v$ is achronal (i.e., spacelike or null). If $\mu$ and $\nu$ are not bounded, then $\tau\to\pm \frac{1}{a}$.
\end{lemma}

\begin{proof}
If $\mu$ is bounded (and thus converges to zero by Lemma \ref{lemma4.1}), then from \eqref{eq2.17} it is clear that $\tau$ and $\nu$ are bounded. Since $\nu$ is monotonic, it must converge, and then by \eqref{eq2.17} so must $\tau$.

If instead $\nu$ is bounded, since it's monotonic it must converge, and then as argued in the proof in the previous lemma, we have $\tau\to 0$. Then by \eqref{eq2.17} $\mu$ converges as well. By Lemma \ref{lemma4.1}, then $\mu\to 0$. Hence \eqref{eq2.17} yields $\nu^2\to\eta(v,v)$, which requires that $\eta(v,v)\ge 0$ so $v$ must be achronal.

If instead neither $\mu$ nor $\nu$ are bounded, then write \eqref{eq2.17} as
\begin{equation}
\label{eq4.11}
(\nu+\mu)(\nu-\mu)=\tau^2+\eta(v,v)\ .
\end{equation}
Since $\tau$ is bounded, so is the product $(\nu+\mu)(\nu-\mu)$. But neither $\mu$ nor $\nu$ is bounded so one factor in this product diverges and so the other factor must converge to zero. Then $\lim_{s\to\infty} \frac{\mu}{\nu}=\pm 1$.

Because $\tau$ is bounded, $\limsup_{s\to\infty} \tau(s)$ and $\liminf_{s\to\infty} \tau(s)$ exist. Now $\limsup_{s\to\infty} \tau(s)=\lim_{\alpha\to\infty}\tau(s_{\alpha})$ where each $\tau(s_{\alpha})$ is a local maximum of $\tau$. At each local maximum, we have by the first equation in \eqref{eq2.18} that $\tau(s_{\alpha})=-\frac{\mu(s_{\alpha})}{a\nu(s_{\alpha})}$, so $\limsup_{s\to\infty} \tau = -\limsup_{s\to\infty} \frac{\mu(s_{\alpha})}{a\nu(s_{\alpha})}=\mp\frac{1}{a}$ using the result of the previous paragraph. But we can replace the local maxima $\tau(s_{\alpha})$ by local minima, say $\tau(s_{\beta})$ and obtain $\liminf_{s\to\infty} \tau =\mp\frac{1}{a}$ (with the same choice of sign). Hence $\lim_{s\to\infty} \tau = \mp \frac{1}{a}$.
\end{proof}

\subsection{Spacelike ${\tilde v}$} Each spacelike vector ${\tilde v}\in{\mathbb M}^3$ is contained in a $1$-parameter family of planes, exactly two of which are null. Boosts in the spacelike directions orthogonal to each such plane preserve the planes and preserve $\pm {\tilde v}$, but do not preserve any spacelike vector in these planes.

We can apply a constant boost to the standard basis of ${\mathbb M}^3$ so that an arbitrary spacelike ${\tilde v}$ has the form ${\tilde v}=({\tilde v}_1,{\tilde v}_2,0)$ in the transformed basis. A subsequent constant rotation of the basis brings $v$ to the form ${\tilde v}=a(0,1,0)$ for some $a>0$. As before, we will normalize and write $v={\tilde v}/a$. Then $v$ lies in the intersection of the null plane $\Span\{(0,1,0),(1,0,1)\}$ with the vertical $x=0$ plane, while $\mu(s)$ as defined by \eqref{eq2.15} is simply the $y$-coordinate of $X(s)$.  The curve of intersection of the $x=0$ plane and the unit hyperboloid is the geodesic in ${\mathbb M}^3$ whose trace is the graph of $z=\sqrt{1+y^2}$, $x=0$. The intersection of the unit hyperboloid and the plane $y=0$ orthogonal to $v$ is also a geodesic, which we denote by $\Gamma_v$. We will label the $y>0$ half-space (into which $v$ points) in ${\mathbb M}^3$ by $H_+$ and the $y<0$ half-space by $H_-$.

\begin{lemma}\label{lemma4.6}
Let $X(s)$ be a non-geodesic curve evolving by isometries under CSF with spacelike ${\tilde v}$ and let $v={\tilde v}/\sqrt{\eta({\tilde v},{\tilde v})}$. Then exactly one of the following holds:
\begin{itemize}
\item [(i)] $\mu(s)$ has exactly one critical point $s_0$, a minimum, $X(s)$ lies entirely in $H_+$, and $X(s)$ extends to infinity at both ends. See Figure \ref{figure1}.
\item [(ii)] $\mu(s)$ has exactly one critical point $s_0$, a maximum, $X(s)$ lies entirely in $H_-$,  and $X(s)$ extends to infinity at both ends. See Figure \ref{figure2}.
\item[(iii)] $\mu(s)$ has no critical point, and either converges to $\Gamma_v$ at one end while extending to infinity at the other or intersects $\Gamma_v$ at one point and extends to infinity at both ends. See Figures \ref{figure3} and \ref{figure4}.
\end{itemize}
\end{lemma}

\begin{proof} If $\mu$ has a critical point in $H_+$, by Lemma \ref{lemma4.1} it is a global minimum. Then $\mu(s)\ge \mu(s_0)$ for all $s$, so the curve must remain in $H_+$. Likewise if $\mu$ has a critical point in $H_-$, is it a global maximum, and so the curve must remain in $H_-$. In either case, $\mu$ is bounded away from zero so by Lemma \ref{lemma4.1} $\mu\to \infty$ at both ends $s\to\pm\infty$ if the curve is in $H_+$, and $\mu\to -\infty$ at both ends $s\to\pm\infty$ if the curve is in $H_-$.

If a critical point $s_0$ of $\mu$ were to lie on $\Gamma_v$, then $\mu(s_0)=\mu'(s_0)=0$. Then $\tau(s_0)=0$ and $\nu(s_0)=\pm 1$. (We can further deduce that $\tau'(s_0)=0$ and $\mu''(s_0)=0$.) For these initial data, the unique maximal solution of the system \eqref{eq2.18} has $\tau(s)=0$ for all $s$, and thus $\kappa_g(s)=a\tau(s)=0$ for all $s$, so the corresponding curve is a geodesic, contrary to assumption. This possibility therefore does not occur.

If a curve has no critical point of $\mu$ then $\mu$ is monotonic so $\mu$ has at most one zero. Then the curve meets $\Gamma_v$ in at most one point. If it does not meet $\Gamma_v$, then it must approach $\Gamma_v$ in the limit, since it cannot have a limit for $\mu$ other than $\mu\to 0$.
\end{proof}

\begin{figure}[!ht]
\begin{center}
   \resizebox{2cm}{!}{\includegraphics{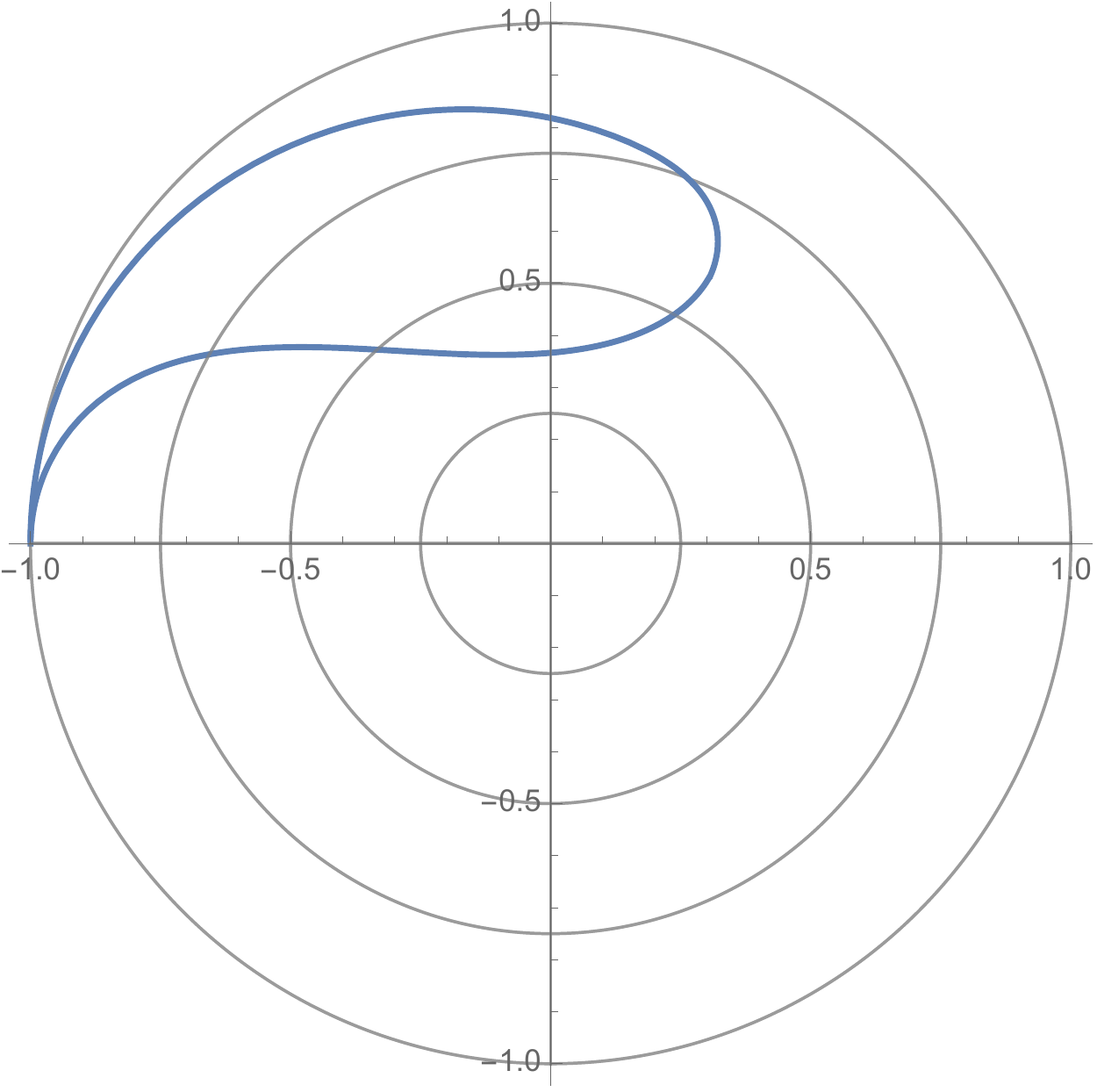}}
   \resizebox{2cm}{!}{\includegraphics{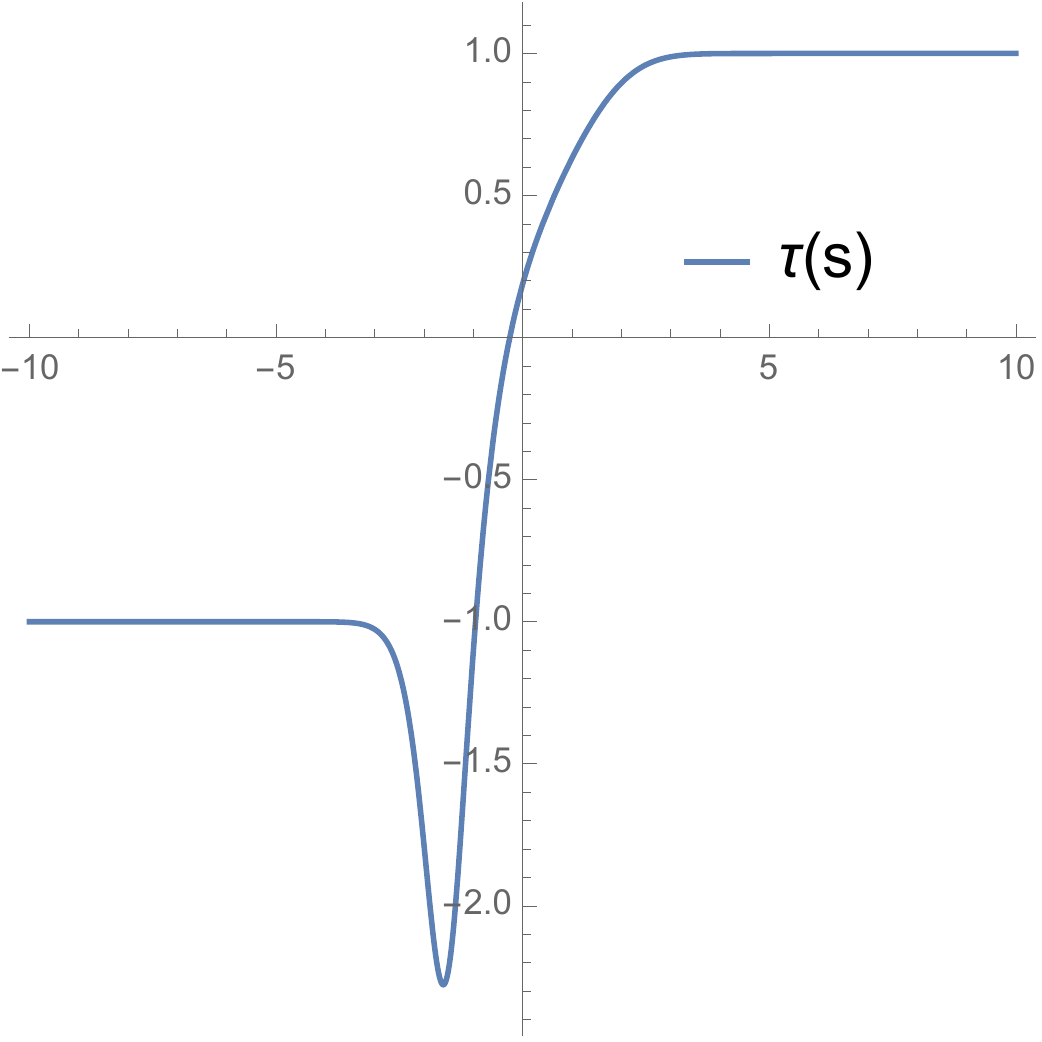}}
   \resizebox{2cm}{!}{\includegraphics{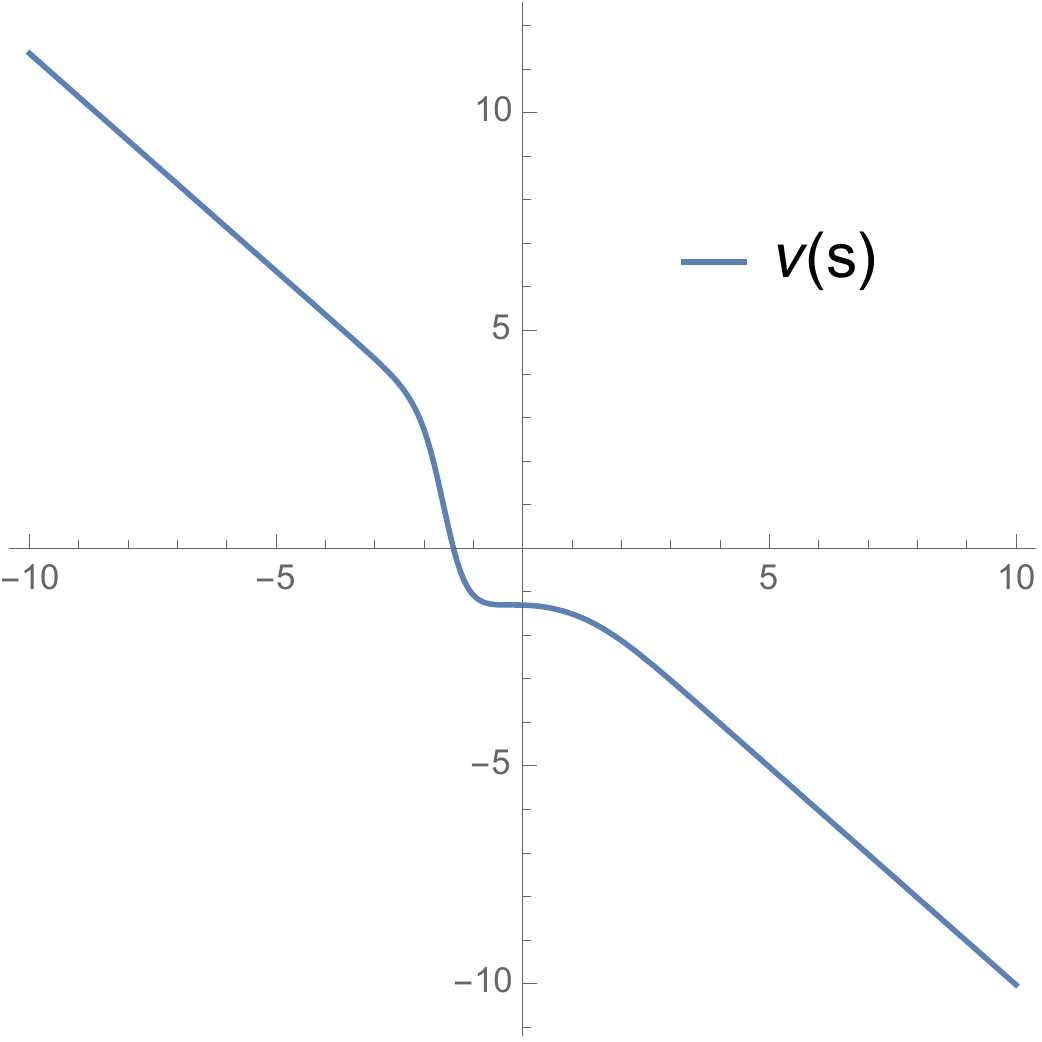}}
   \resizebox{2cm}{!}{\includegraphics{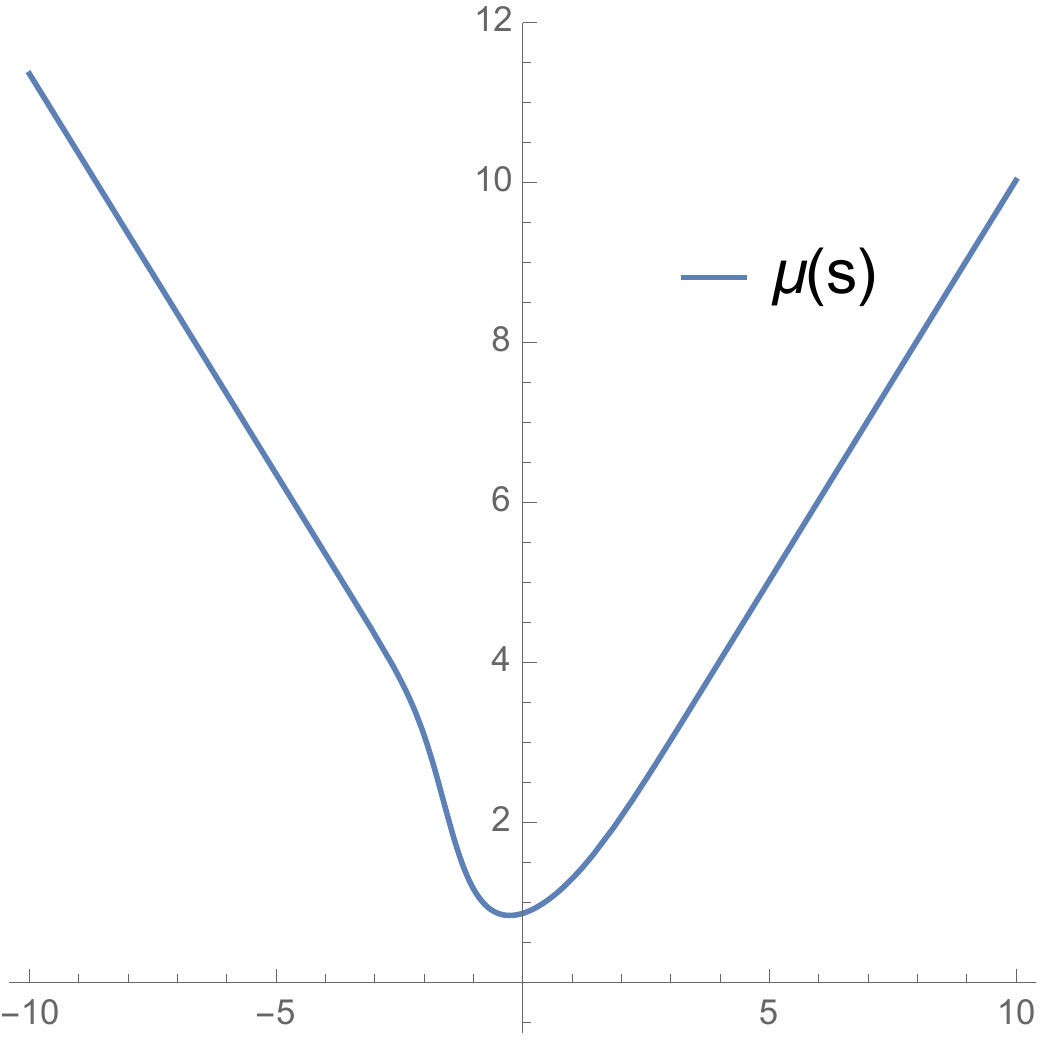}}
   \caption{Curve in the Poincar\'e disk model, evolving by isometry with spacelike ${\tilde v}$, as in Lemma \ref{lemma4.6}.(i). The valley in the graph of $\tau$ indicates the maximum of $|\kappa_g|$.}\label{figure1}
  \end{center}
\end{figure}

\begin{figure}[!ht]
\begin{center}
   \resizebox{2cm}{!}{\includegraphics{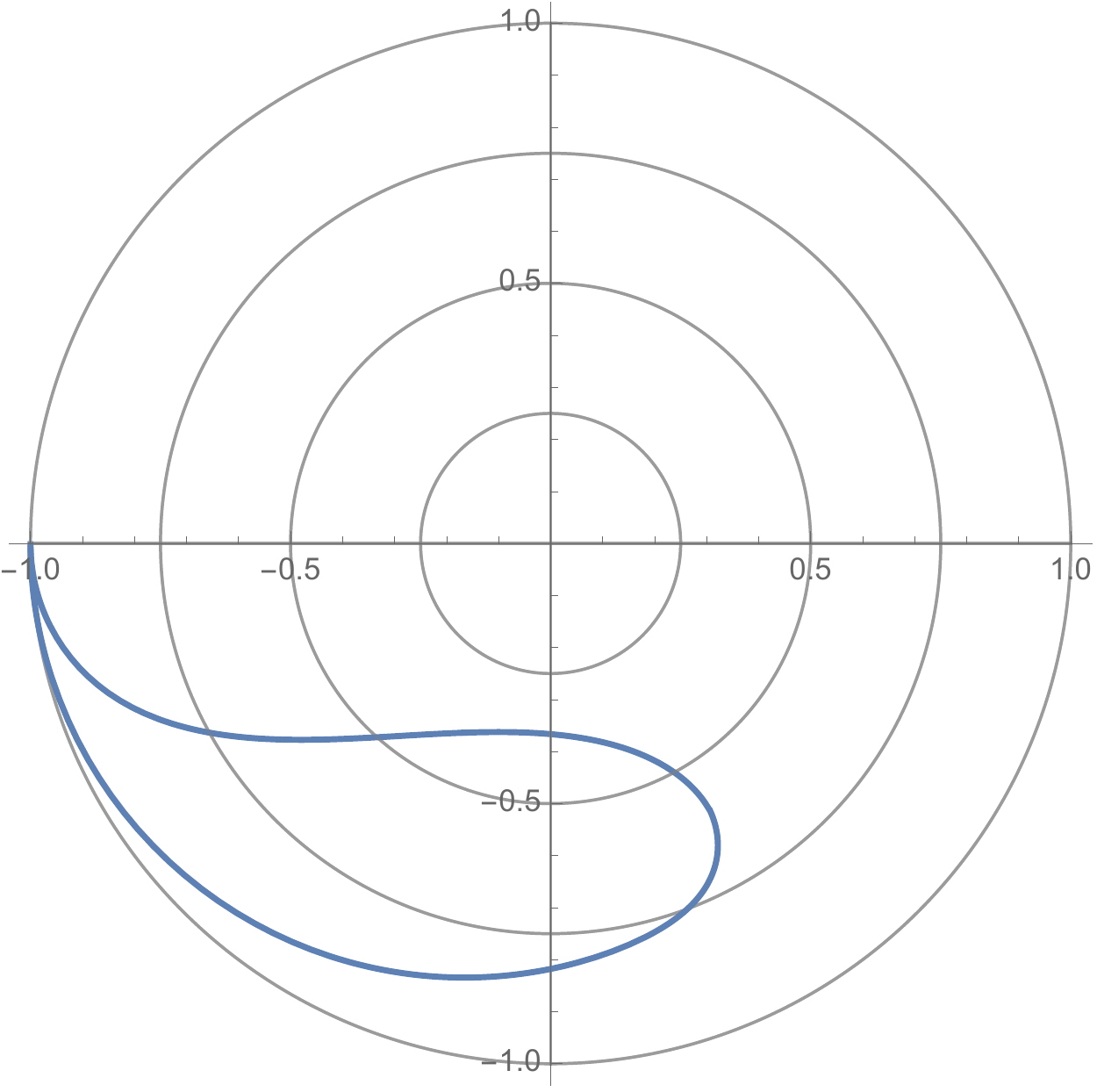}}
   \resizebox{2cm}{!}{\includegraphics{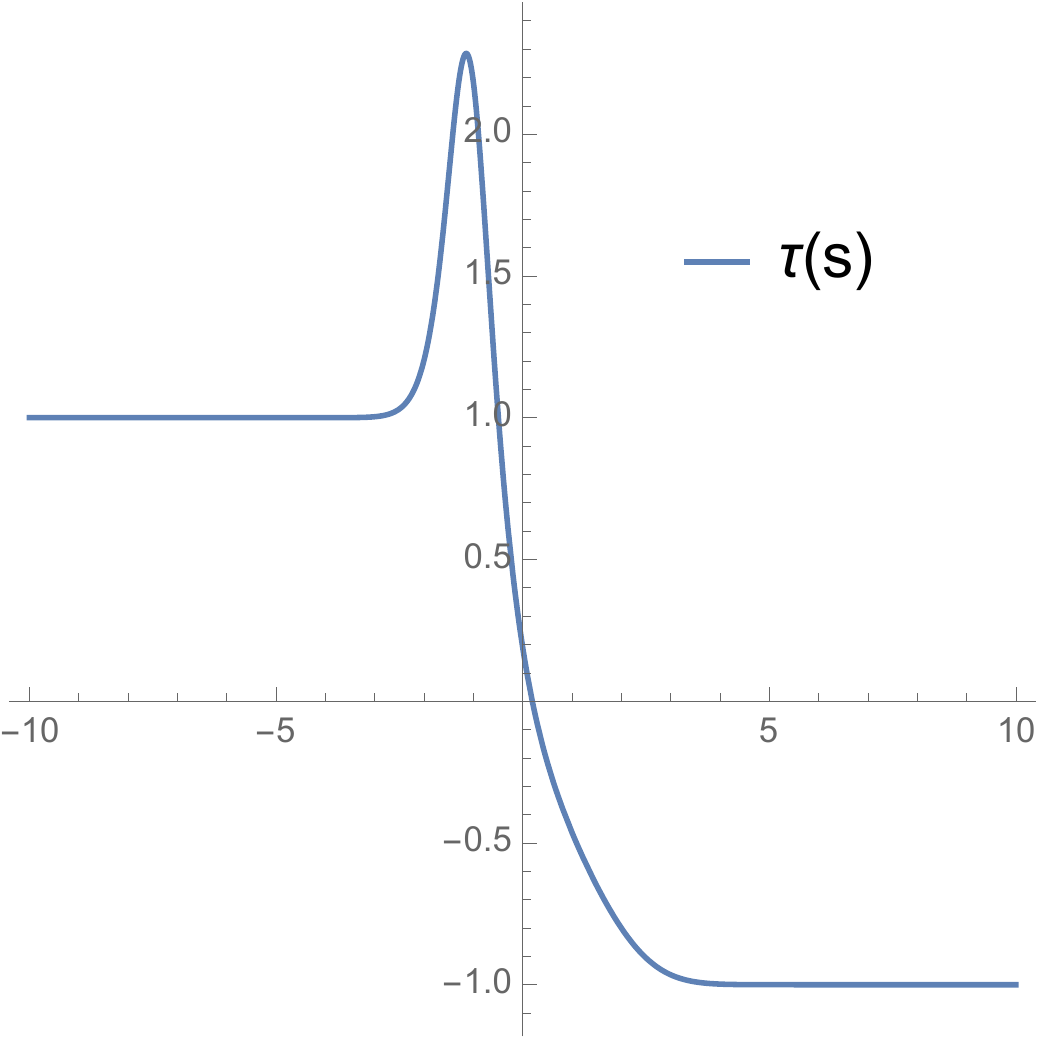}}
   \resizebox{2cm}{!}{\includegraphics{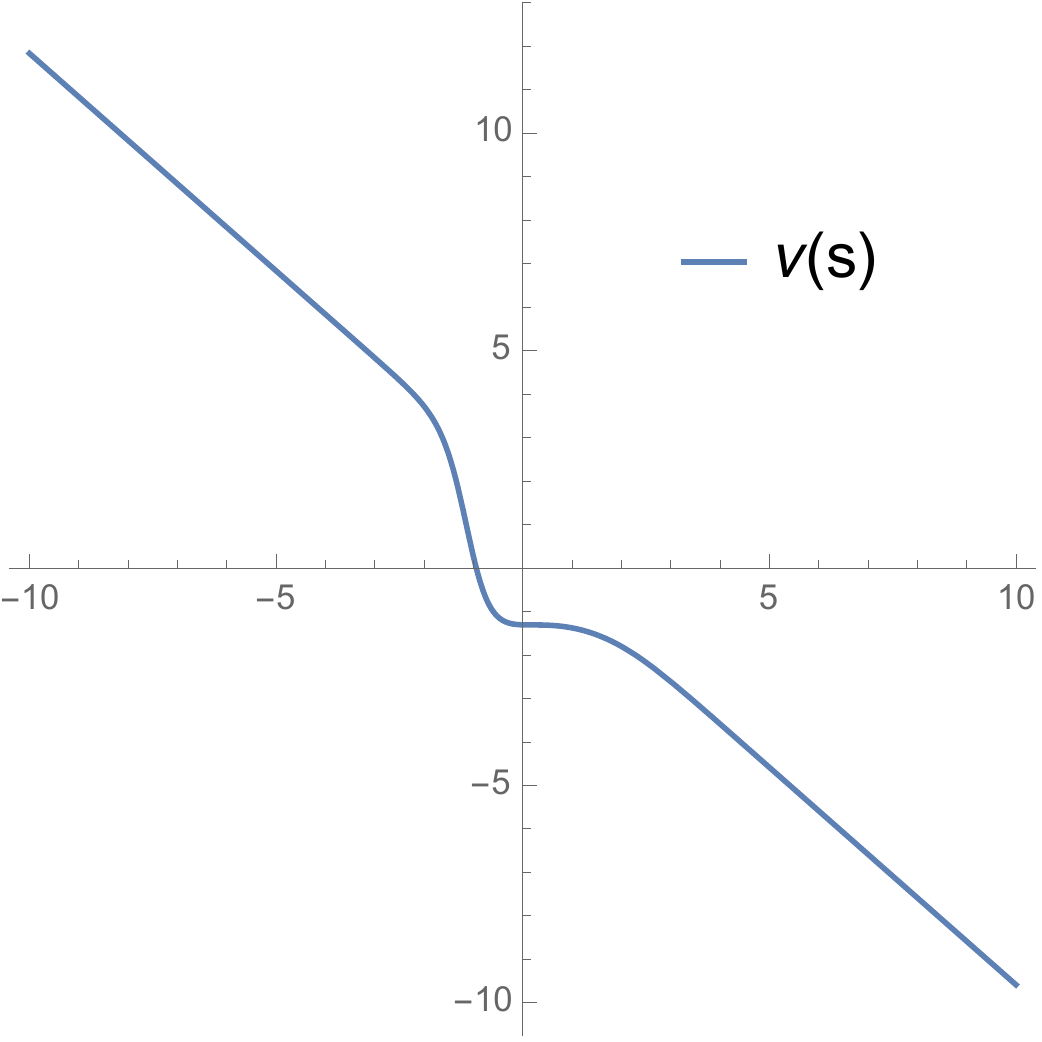}}
   \resizebox{2cm}{!}{\includegraphics{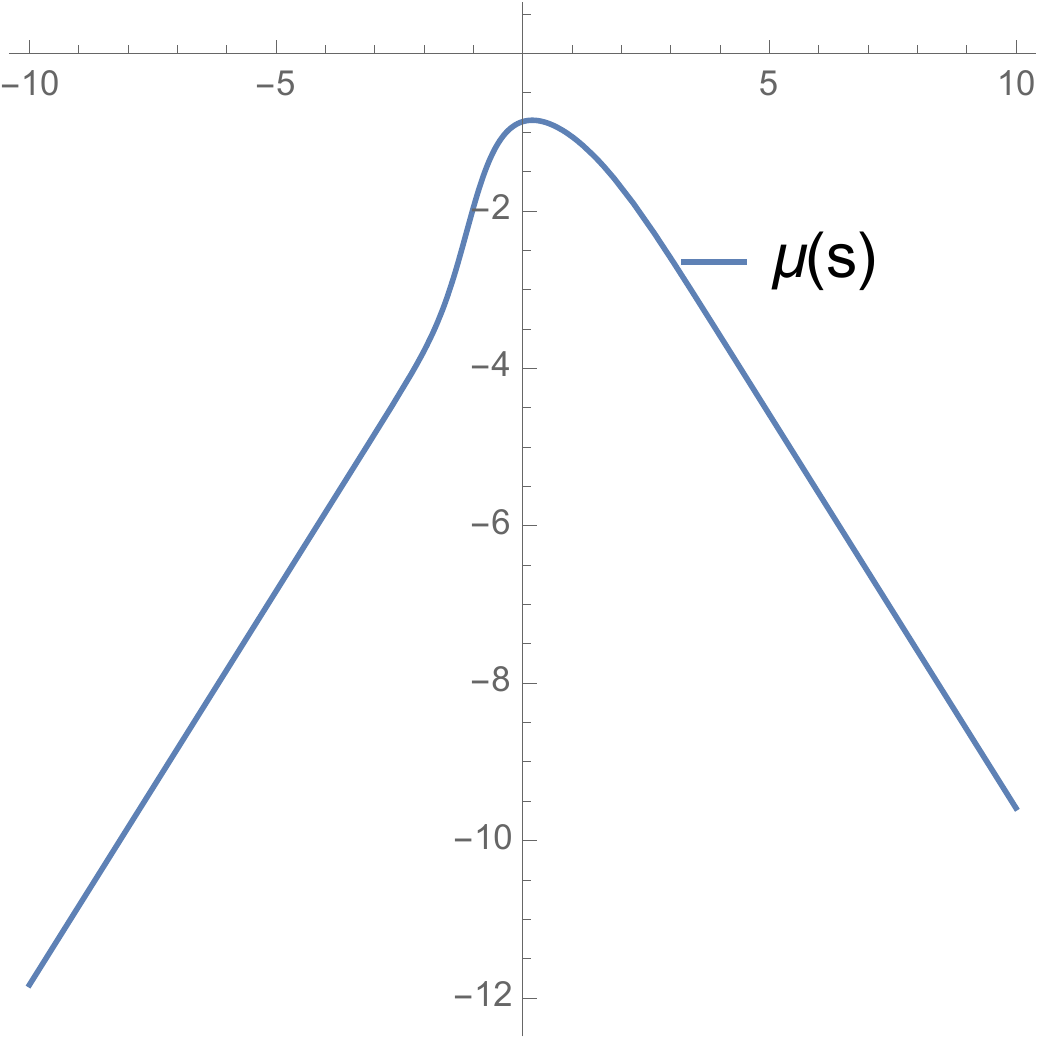}}
   \caption{Curve in the Poincar\'e disk model evolving by isometries with spacelike ${\tilde v}$, as in Lemma \ref{lemma4.6}.(ii).}\label{figure2}
  \end{center}
\end{figure}

\begin{figure}[!ht]
\begin{center}
   \resizebox{2cm}{!}{\includegraphics{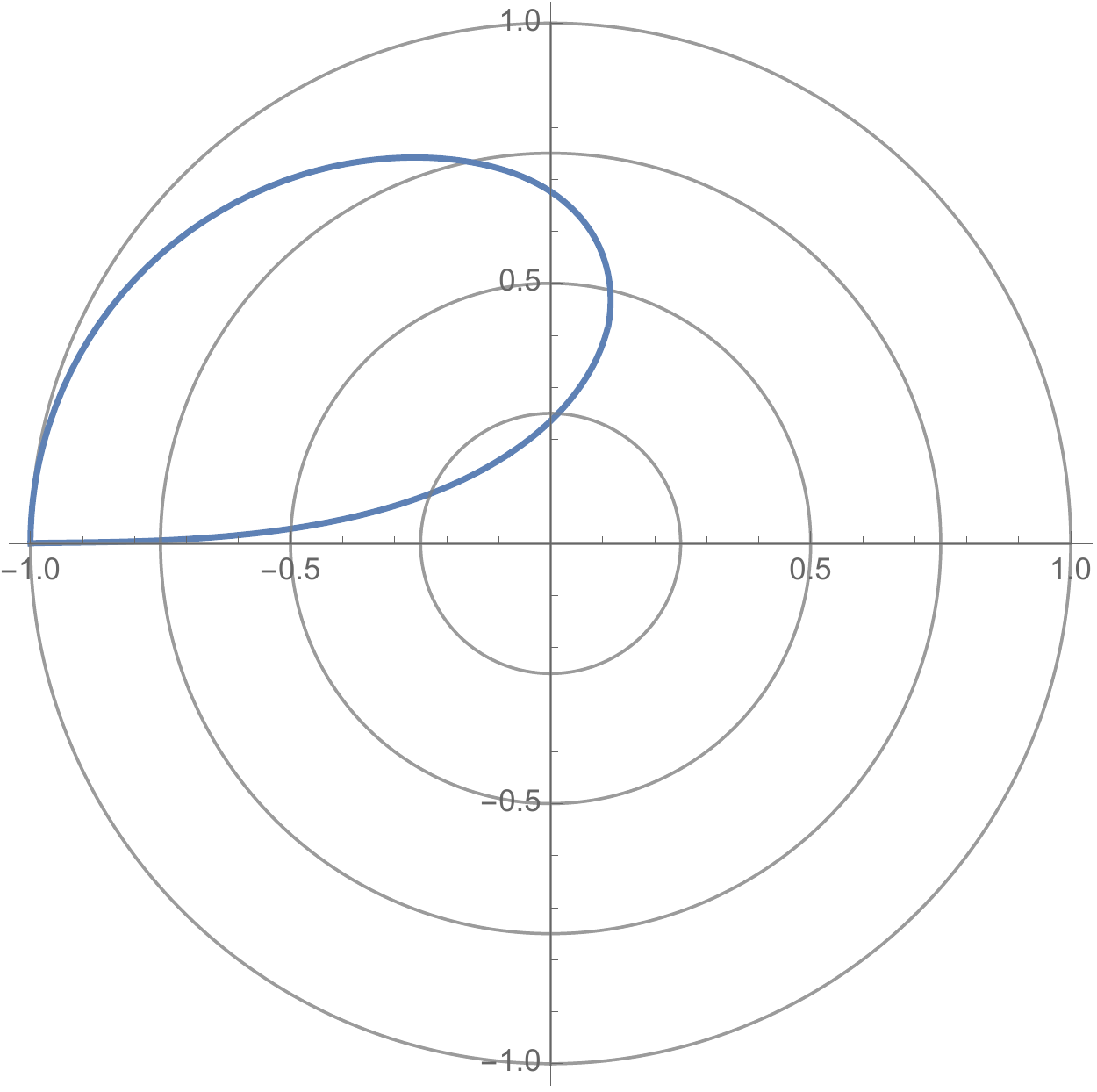}}
   \resizebox{2cm}{!}{\includegraphics{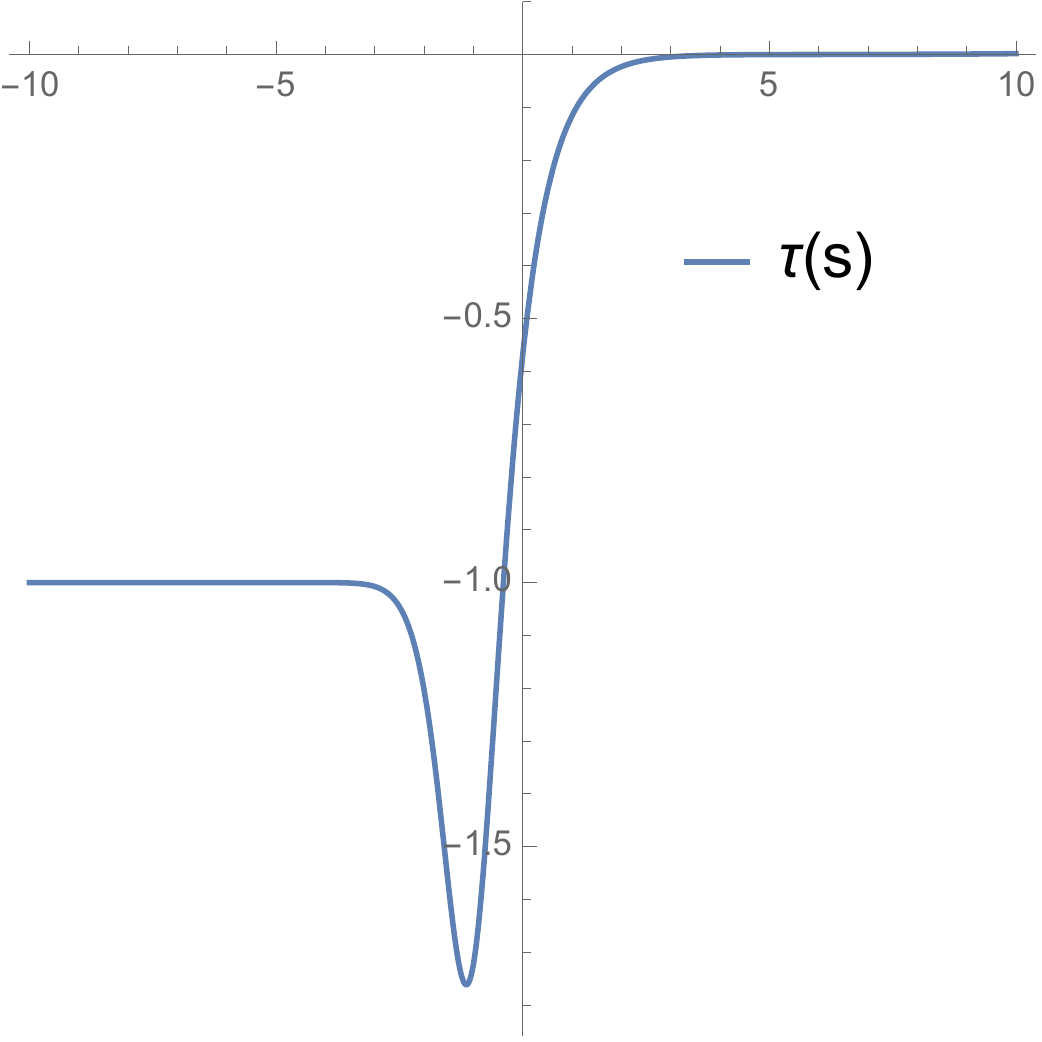}}
   \resizebox{2cm}{!}{\includegraphics{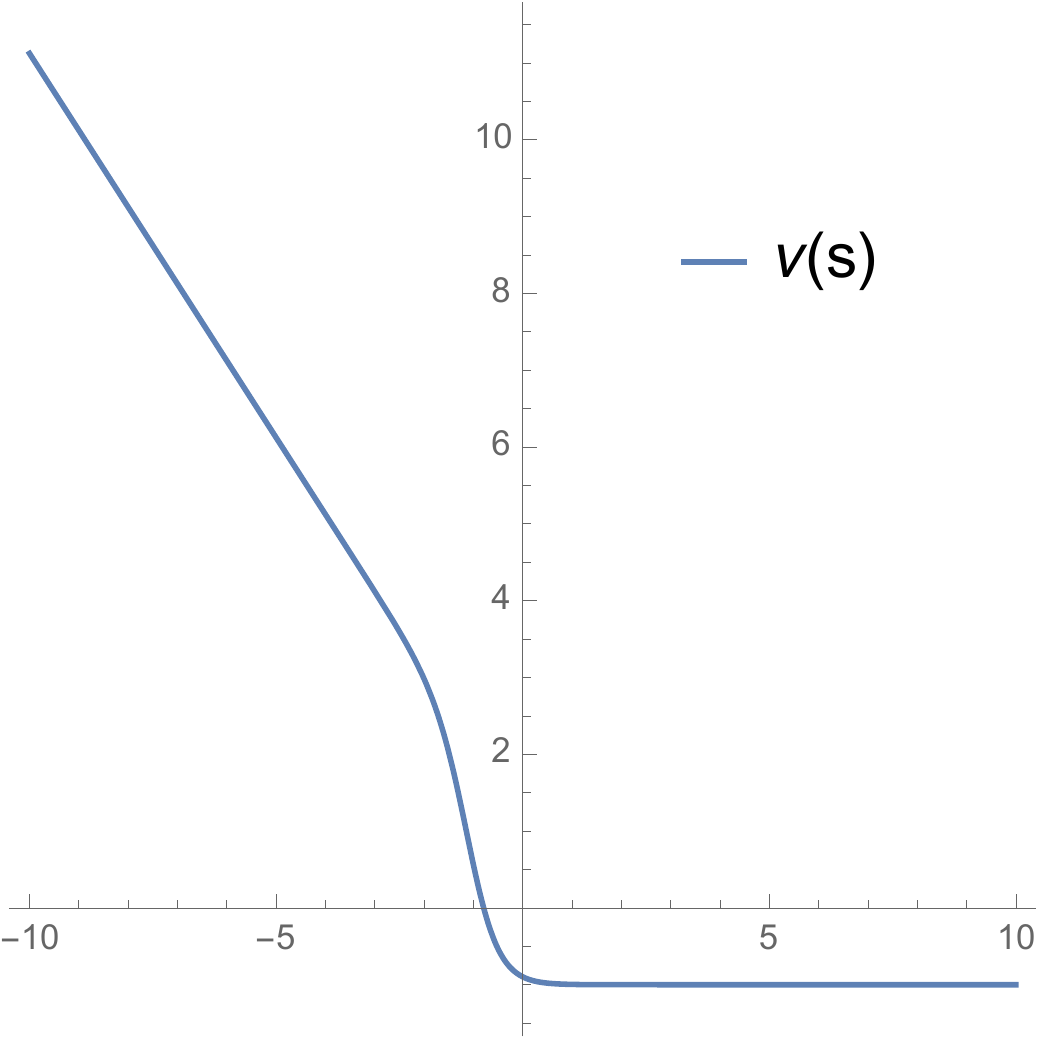}}
   \resizebox{2cm}{!}{\includegraphics{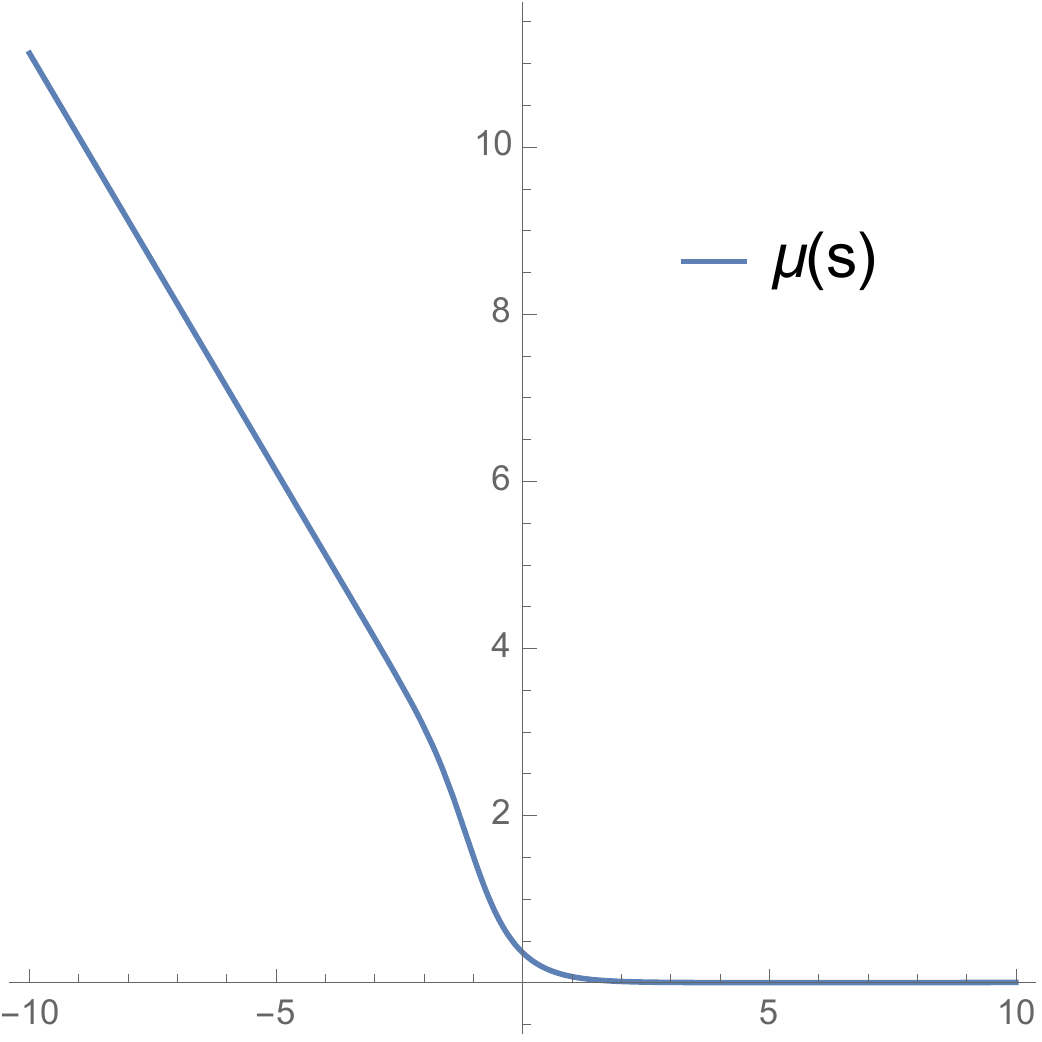}}
   \caption{Curve in the Poincar\'e disk model evolving by isometries with spacelike ${\tilde v}$ and converging to a geodesic at one end, as in Lemma \ref{lemma4.6}.(iii).}\label{figure3}
  \end{center}
\end{figure}

\begin{figure}[!ht]
\begin{center}
   \resizebox{2cm}{!}{\includegraphics{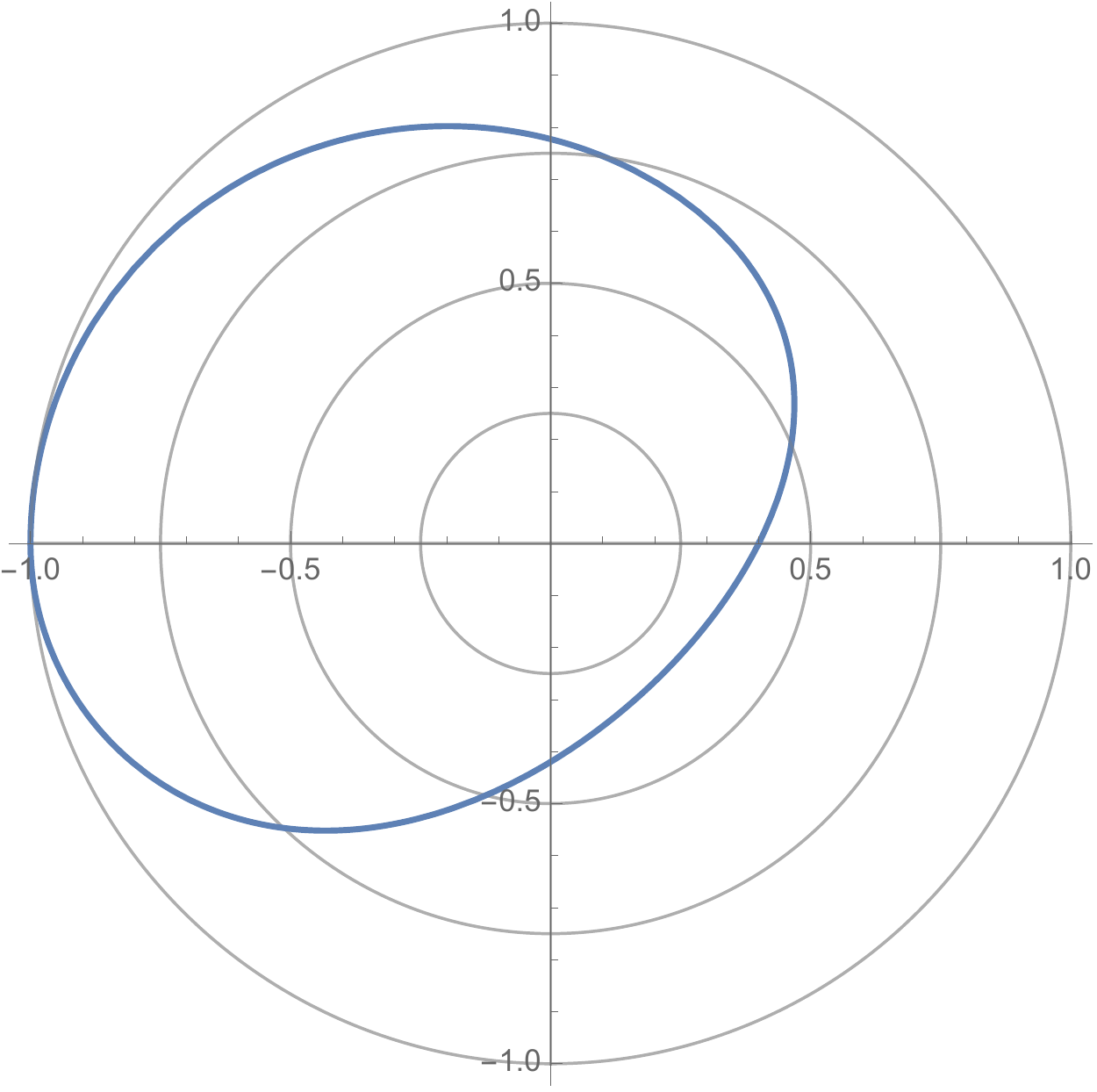}}
   \resizebox{2cm}{!}{\includegraphics{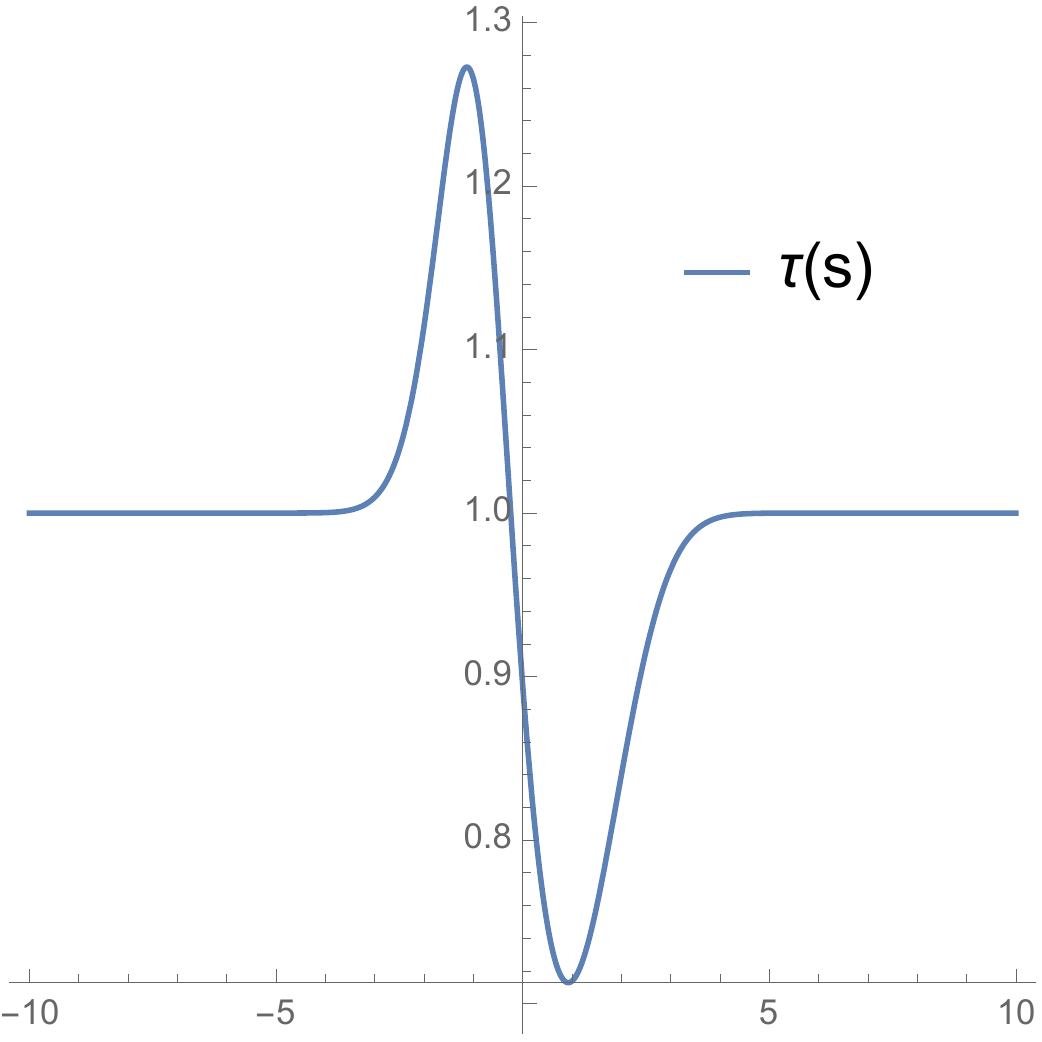}}
   \resizebox{2cm}{!}{\includegraphics{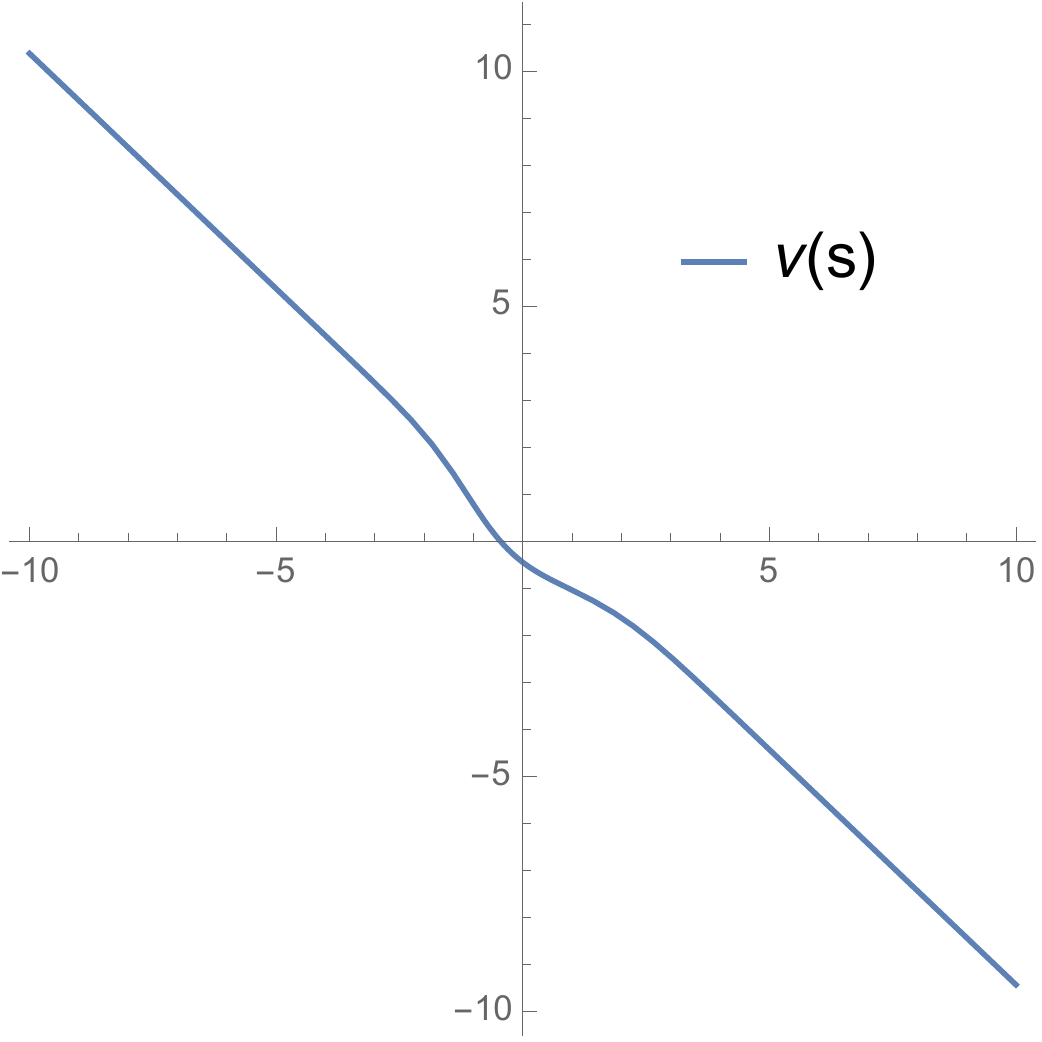}}
   \resizebox{2cm}{!}{\includegraphics{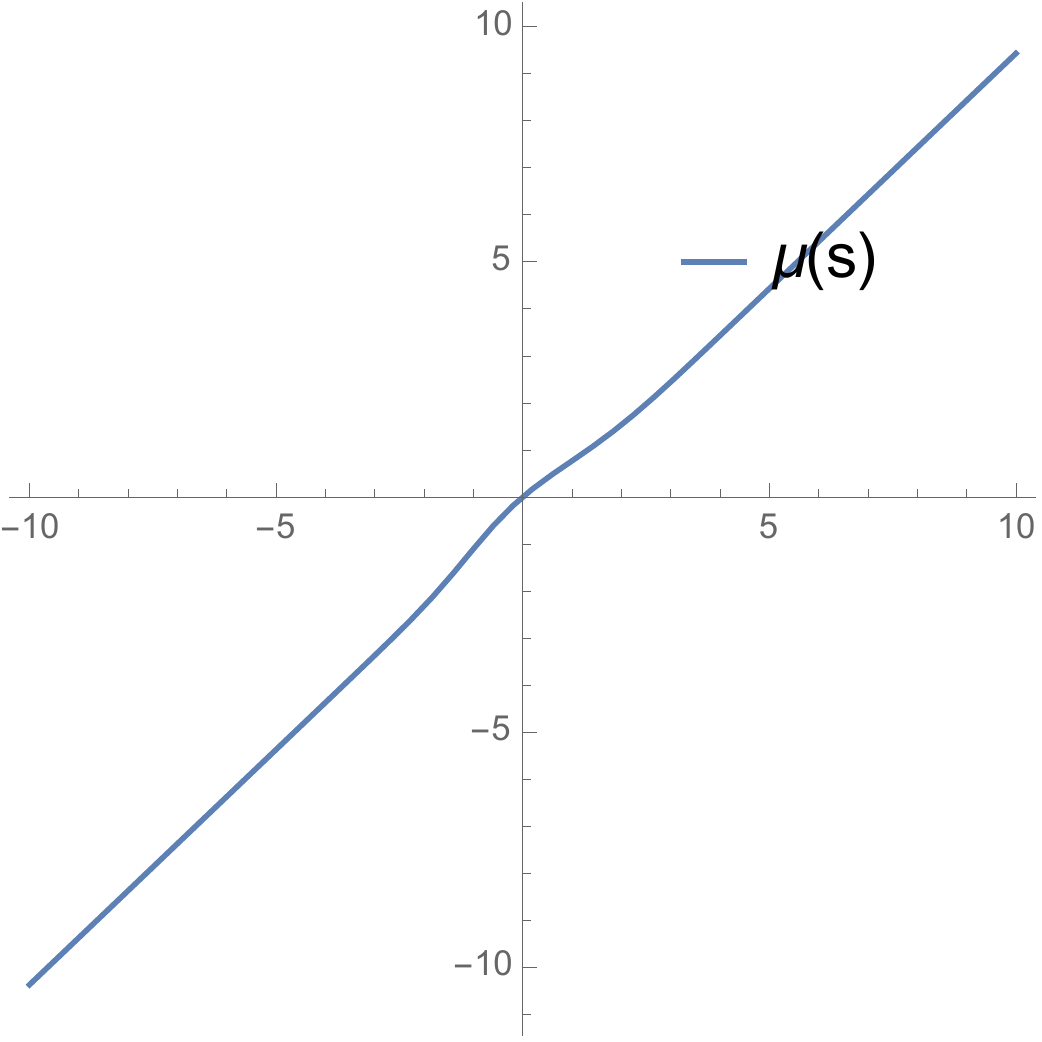}}
   \caption{Curve in the Poincar\'e disk model of hyperbolic space that evolves by isometry with spacelike $v$ corresponding to Lemma \ref{lemma4.6}.(iii), such that $\mu$ has a zero and diverges at both ends.}\label{figure4}
  \end{center}
\end{figure}

\subsection{Timelike ${\tilde v}$} For a timelike ${\tilde v}$, one can apply a constant boost to the axes of the standard basis of ${\mathbb M}^3$ so that in the boosted basis ${\tilde v}$ takes the form ${\tilde v}=(0,0,{\dot \theta})={\dot \theta}(0,0,1)={\dot \theta}v$. In this case, $\mu$ is the negative of the $z$-coordinate of the curve. We've chosen the sign so that $v$ is future-timelike, though ${\tilde v}$ can be either future- or past-timelike depending on the sign of ${\dot \theta}$.

\begin{lemma}\label{lemma4.7}
Let $X(s)$ be a non-geodesic curve evolving by isometries under CSF with future-timelike $v$. Then there is exactly one critical point of $\mu(s)$ along $X(s)$, a global maximum, and $\mu\to -\infty$ as $s\to\pm\infty$ along $X(s)$. See Figures \ref{figure5} and \ref{figure6}.
\end{lemma}

\begin{proof}
Since both $X$ and $v$ are future-timelike and nonzero, then $\mu(s)\le c< 0$ for some $c>0$ and all $s$, so by Lemma \ref{lemma4.1} $\mu$ has a negative global maximum and no other critical point. Since $\mu(s)$ cannot converge to $0$ along $X(s)$, it cannot converge at all and so must diverge to $-\infty$ as $s\to\pm\infty$.
\end{proof}

Since $\mu$ is proportional to $-z$ along a soliton with future-timelike $v$, then $\mu\to\infty$ implies that the curve $X(s)$ extends to infinity along the unit hyperboloid in ${\mathbb M}^3$.

\begin{figure}[!ht]
\begin{center}
   \resizebox{2cm}{!}{\includegraphics{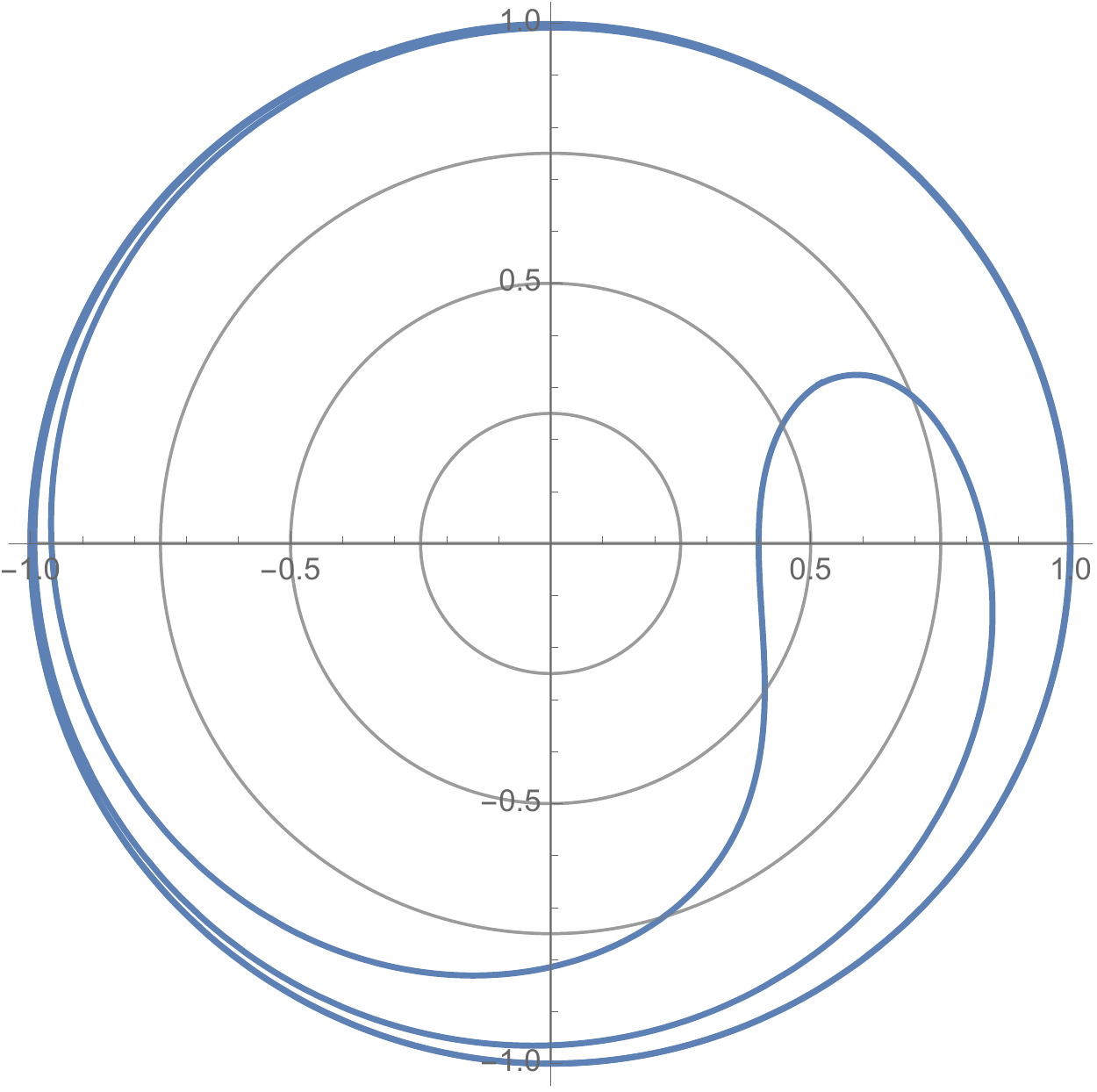}}
   \resizebox{2cm}{!}{\includegraphics{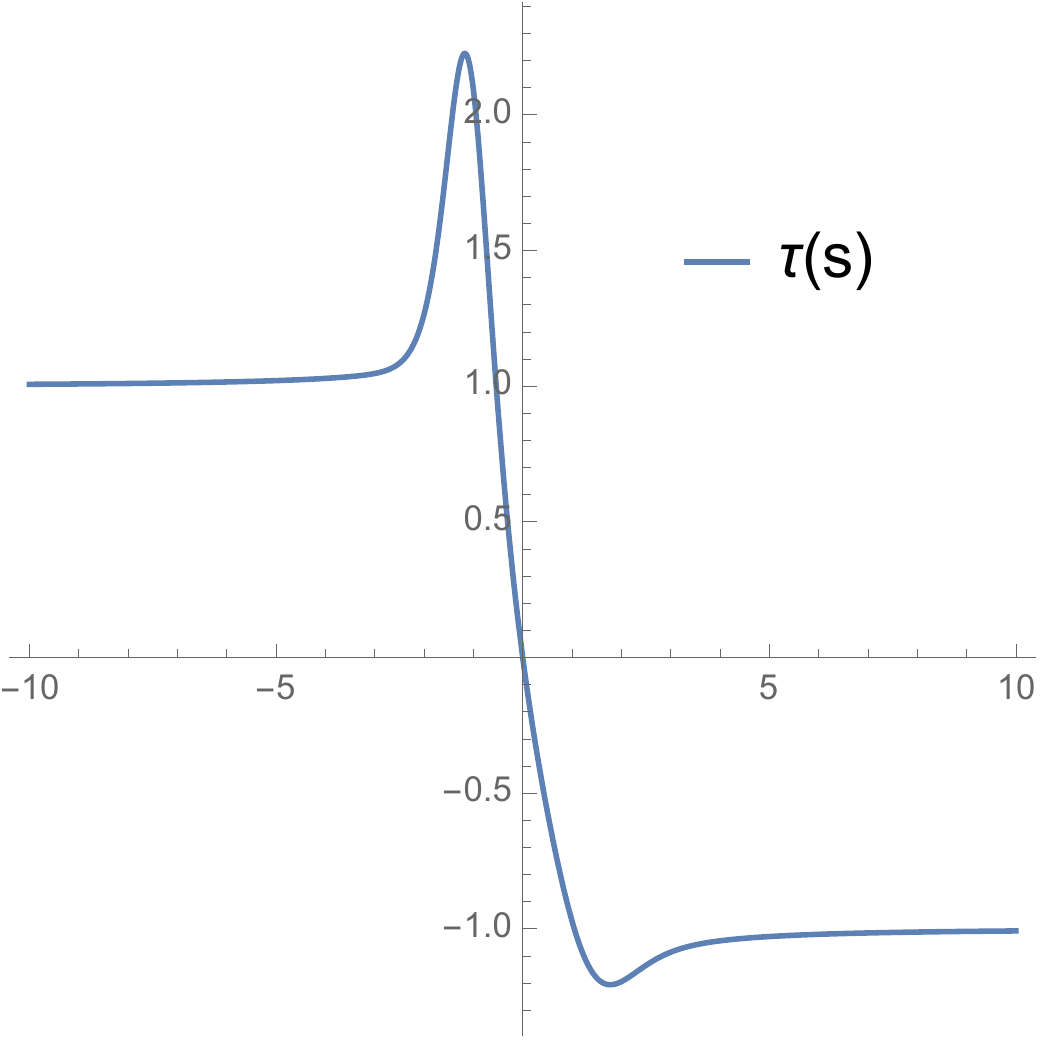}}
   \resizebox{2cm}{!}{\includegraphics{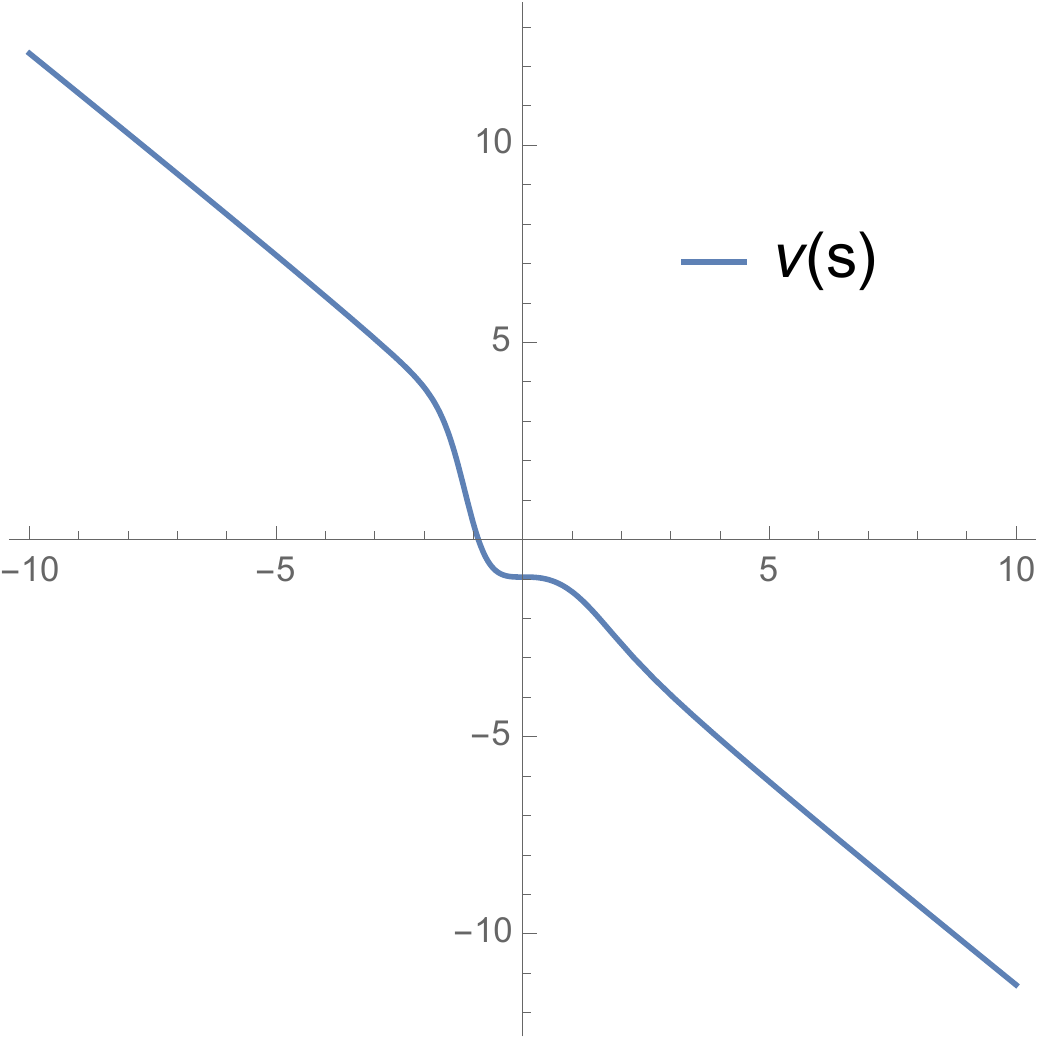}}
   \resizebox{2cm}{!}{\includegraphics{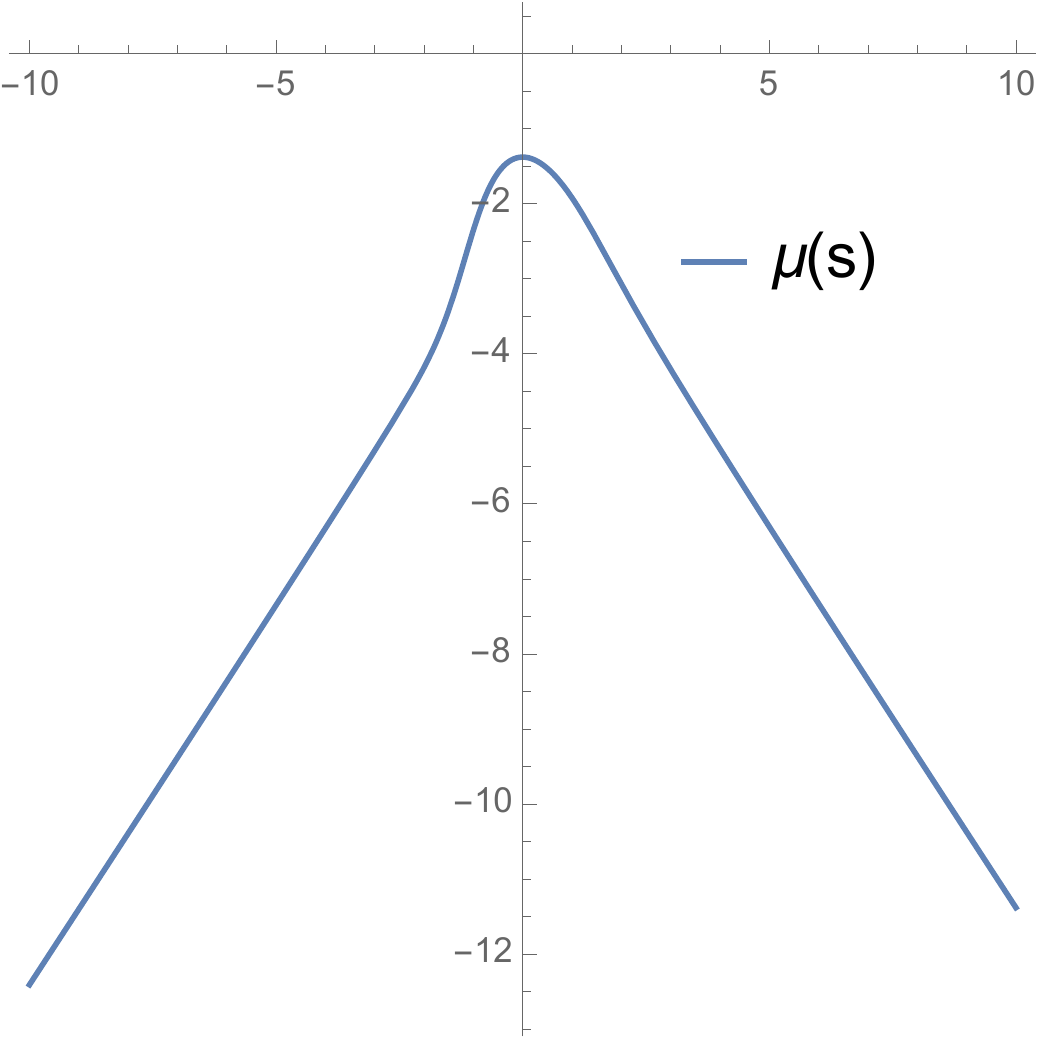}}
   \caption{Curve in the Poincar\'e disk model evolving by isometries with timelike ${\tilde v}$.}\label{figure5}
  \end{center}
\end{figure}

\begin{figure}[!ht]
\begin{center}
   \resizebox{2cm}{!}{\includegraphics{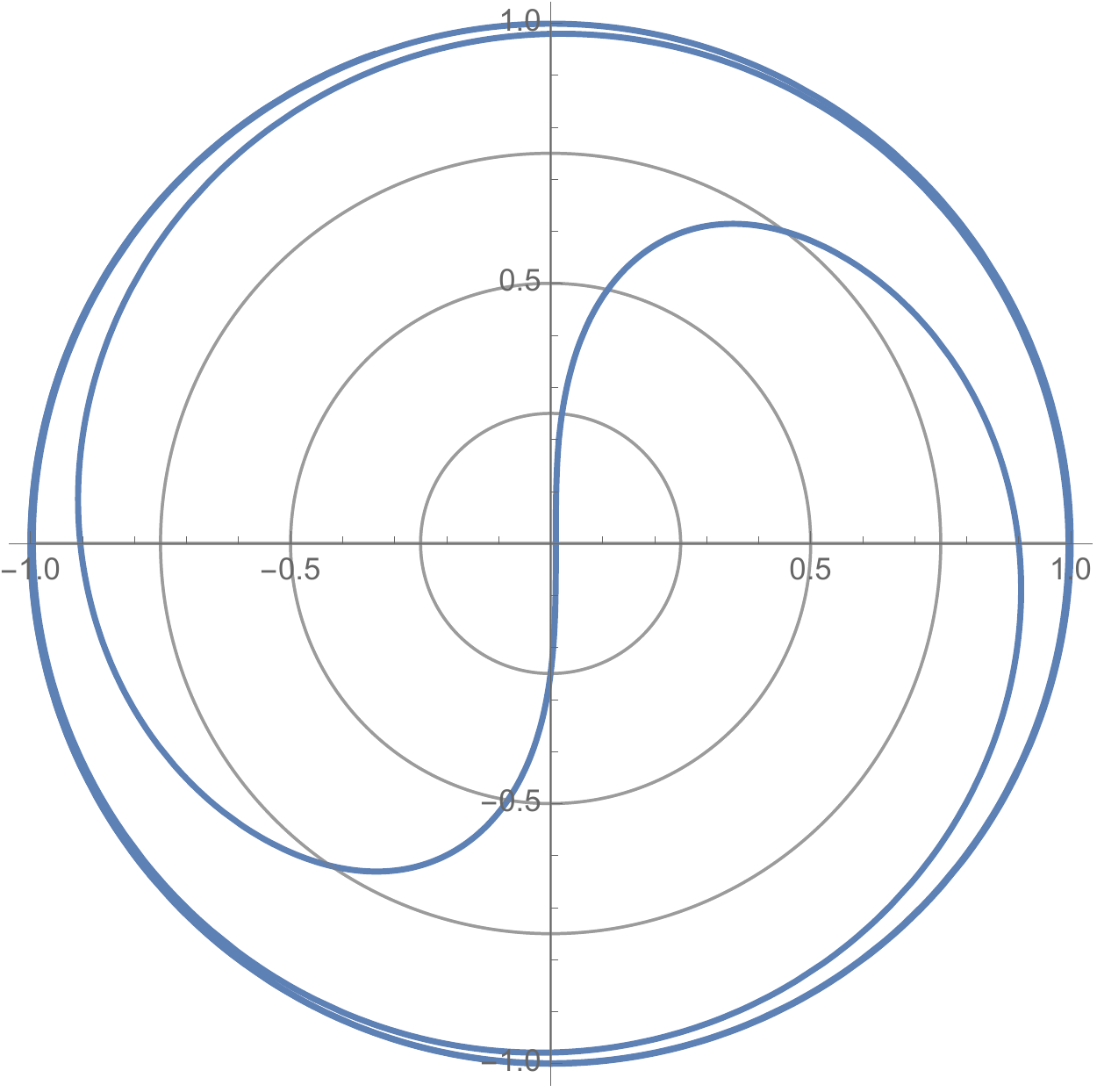}}
   \resizebox{2cm}{!}{\includegraphics{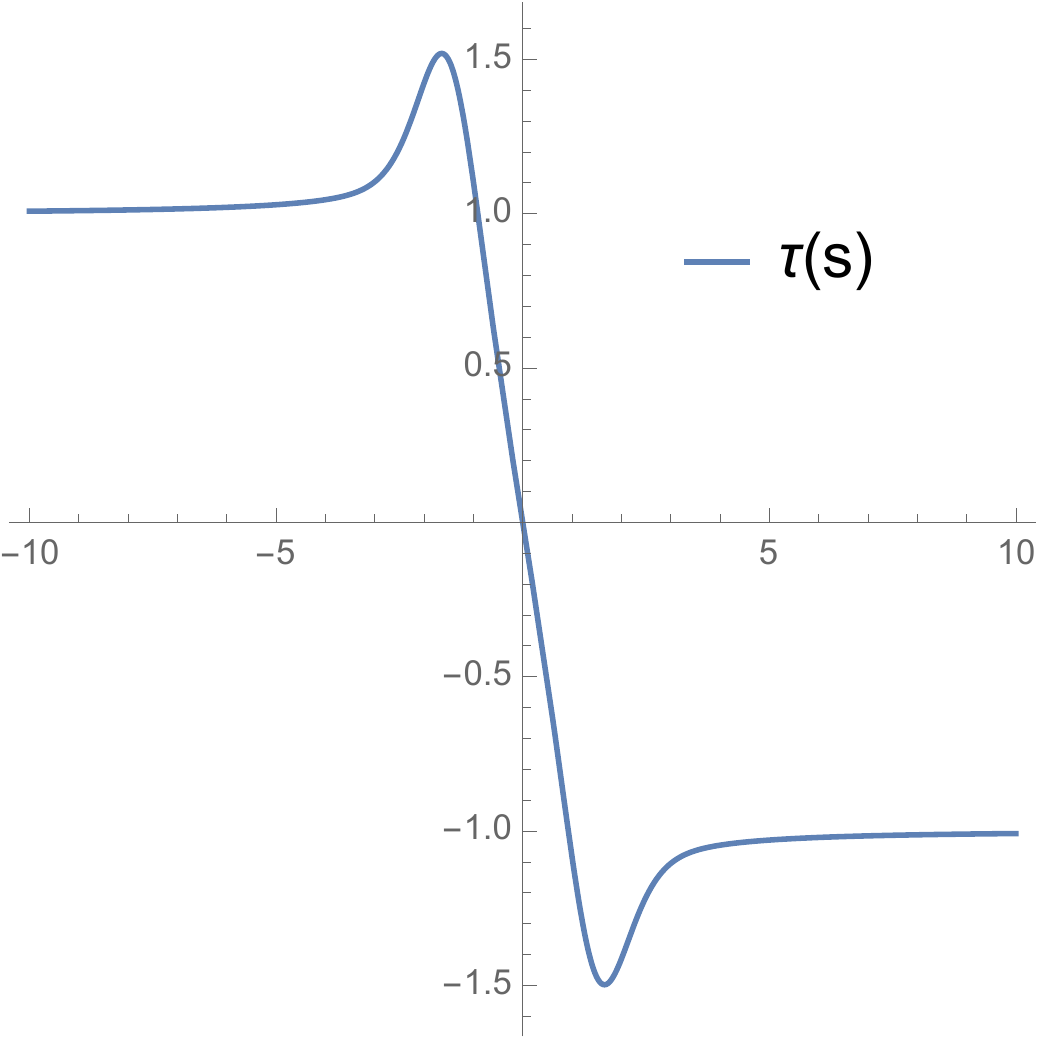}}
   \resizebox{2cm}{!}{\includegraphics{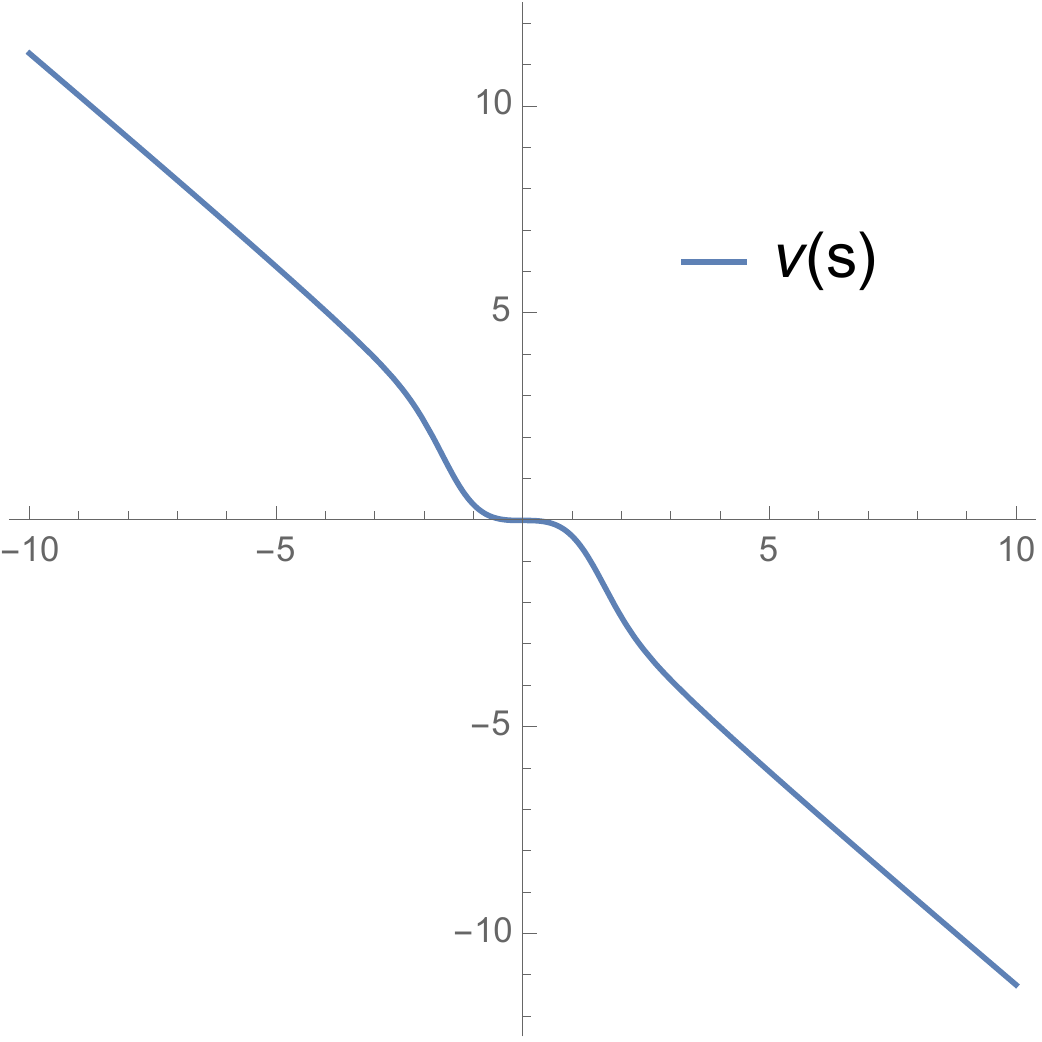}}
   \resizebox{2cm}{!}{\includegraphics{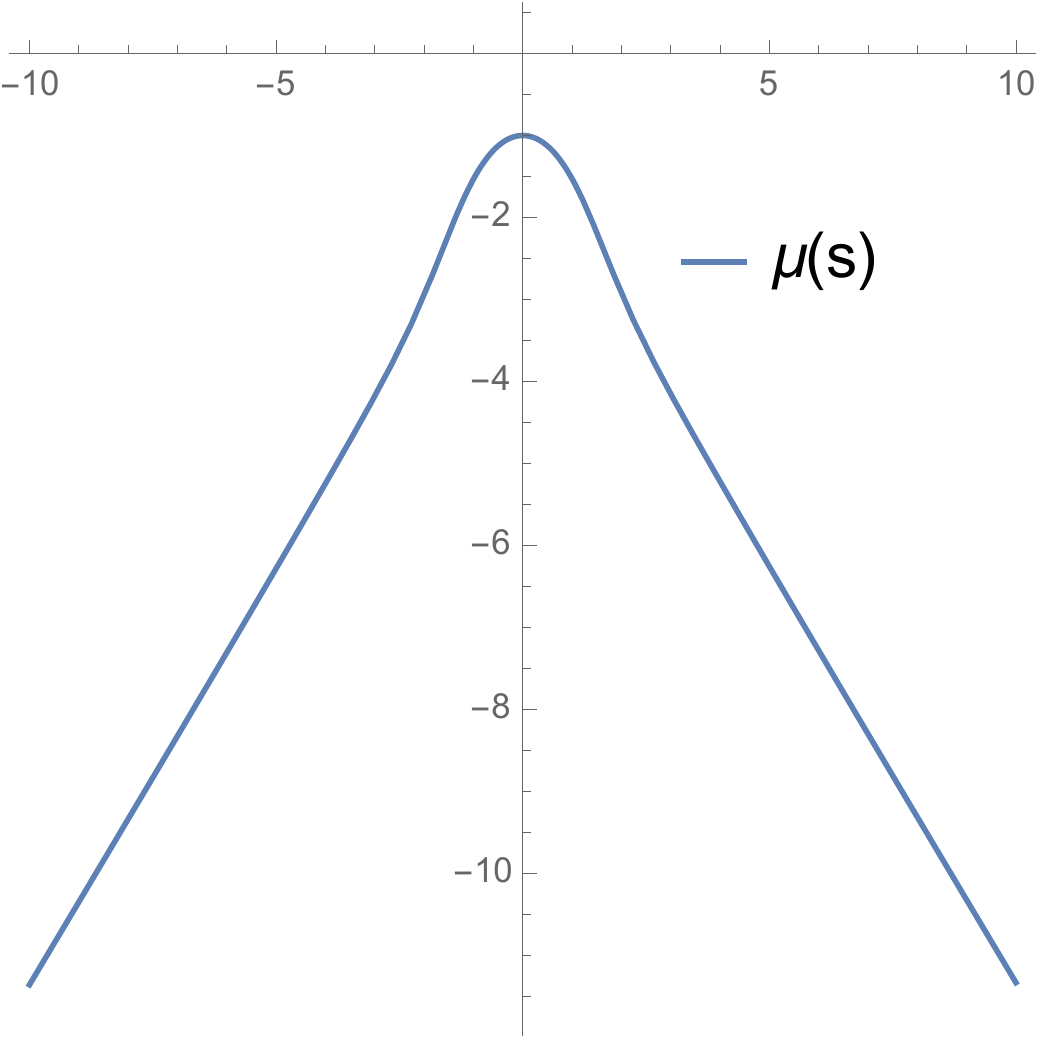}}
   \caption{Curve in the Poincar\'e disk model evolving by isometries with timelike ${\tilde v}$, which passes (nearly) through the origin, where $\mu$ is maximized.}\label{figure6}
  \end{center}
\end{figure}

\subsection{Null ${\tilde v}$} The remaining possibility is that ${\tilde v}$ is null. Applying a constant rotation to the axes if necessary, we can take ${\tilde v}=a(0,1,1)=:av$, so $\mu=y(s)-z(s)<0$ along $X(s)=(x(s),y(s),z(s))$, since $y<z$ at any point on the unit hyperboloid $z=\sqrt{x^2+y^2}$ in ${\mathbb M}^3$.

\begin{lemma}\label{lemma4.8}
Let $X(s)$ be a non-geodesic curve evolving by isometries under CSF with future-null $v$. Then there is at most one critical point of $\mu(s)$ along the curve $X(s)=(x(s),y(s),z(s))$, a global maximum, and $z(s)\to\infty$ at both ends of its domain. See Figure \ref{figure7}.
\end{lemma}

\begin{proof}
Since $y<z$ at each point of the hyperboloid $x^2+y^2-z^2=-1$, $z>0$, then $\mu(s)=y(s)-z(s)<0$ so by Lemma \ref{lemma4.1} there is at most one critical point $s_0$ of $\mu$, and it must be a global maximum. When this occurs, then $\mu(s)\le \mu(s_0)<0$ so $\mu(s)$ cannot converge to zero, and so $\mu\to-\infty$ and $z\to\infty$ as $s\to\pm\infty$.

On the other hand, if there is no such maximum then $\mu$ converges to zero from below as $s\to\infty$ (or as $s\to -\infty$, but not both). But for future-null $v$, $\left \langle v,X\right \rangle \equiv \mu\to 0$ implies that $y-z\to 0$ along the curve. Since $x^2+(y-z)(y+z)=-1$, this can only happen if $y,z\to\pm\infty$.
\end{proof}

For past-null $v$, we take $v\mapsto -v$ so $\mu\mapsto -\mu$ and the global maximum in the lemma becomes a global minimum.

\begin{figure}[!ht]
\begin{center}
   \resizebox{2cm}{!}{\includegraphics{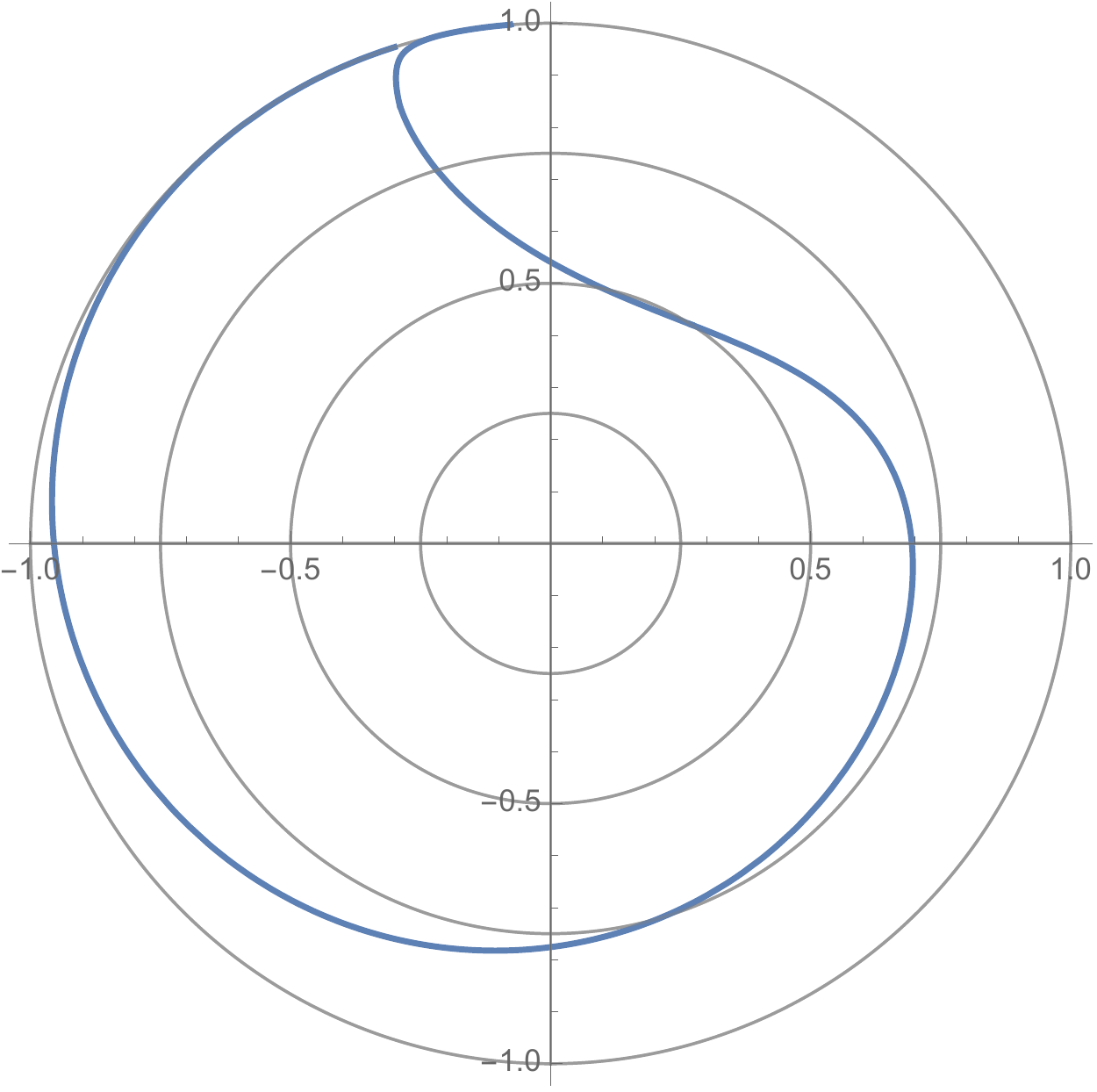}}
   \resizebox{2cm}{!}{\includegraphics{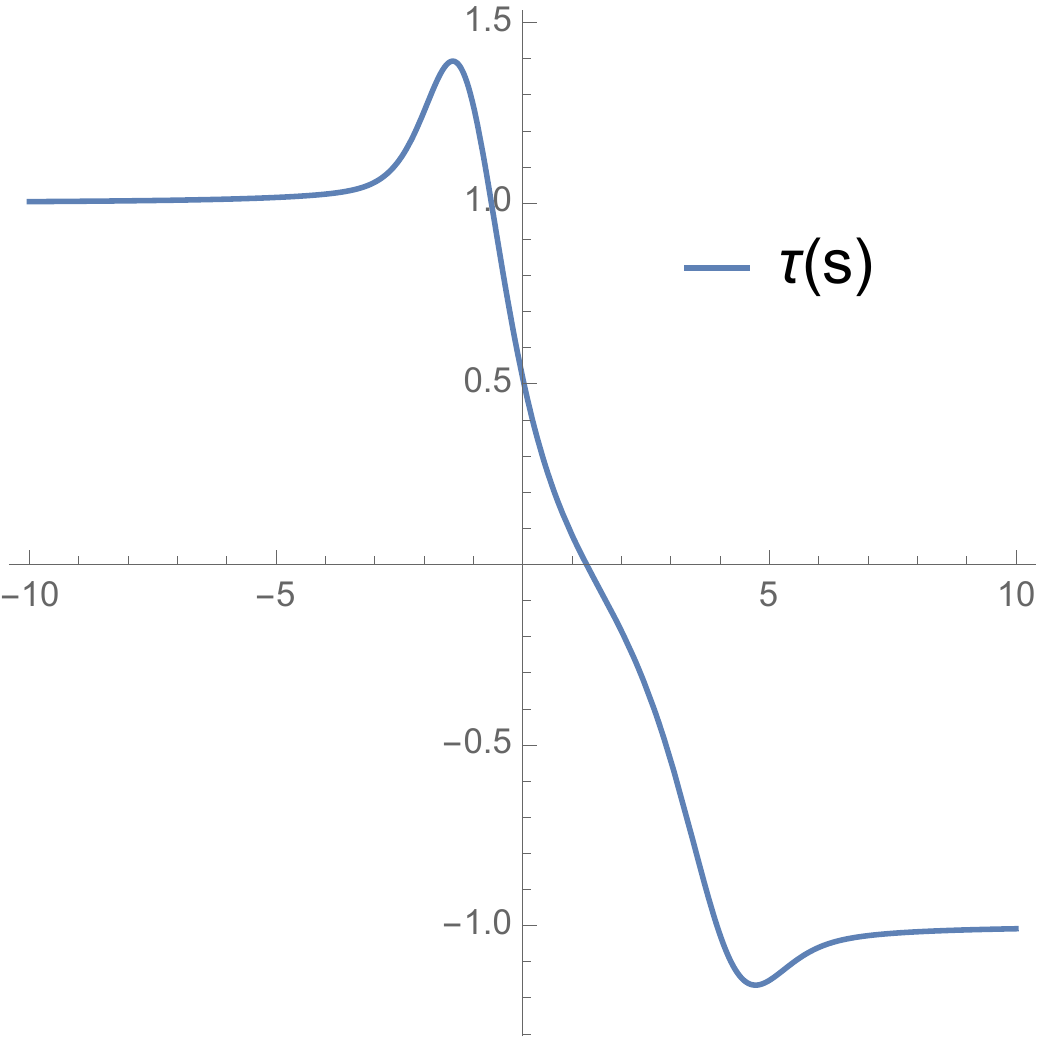}}
   \resizebox{2cm}{!}{\includegraphics{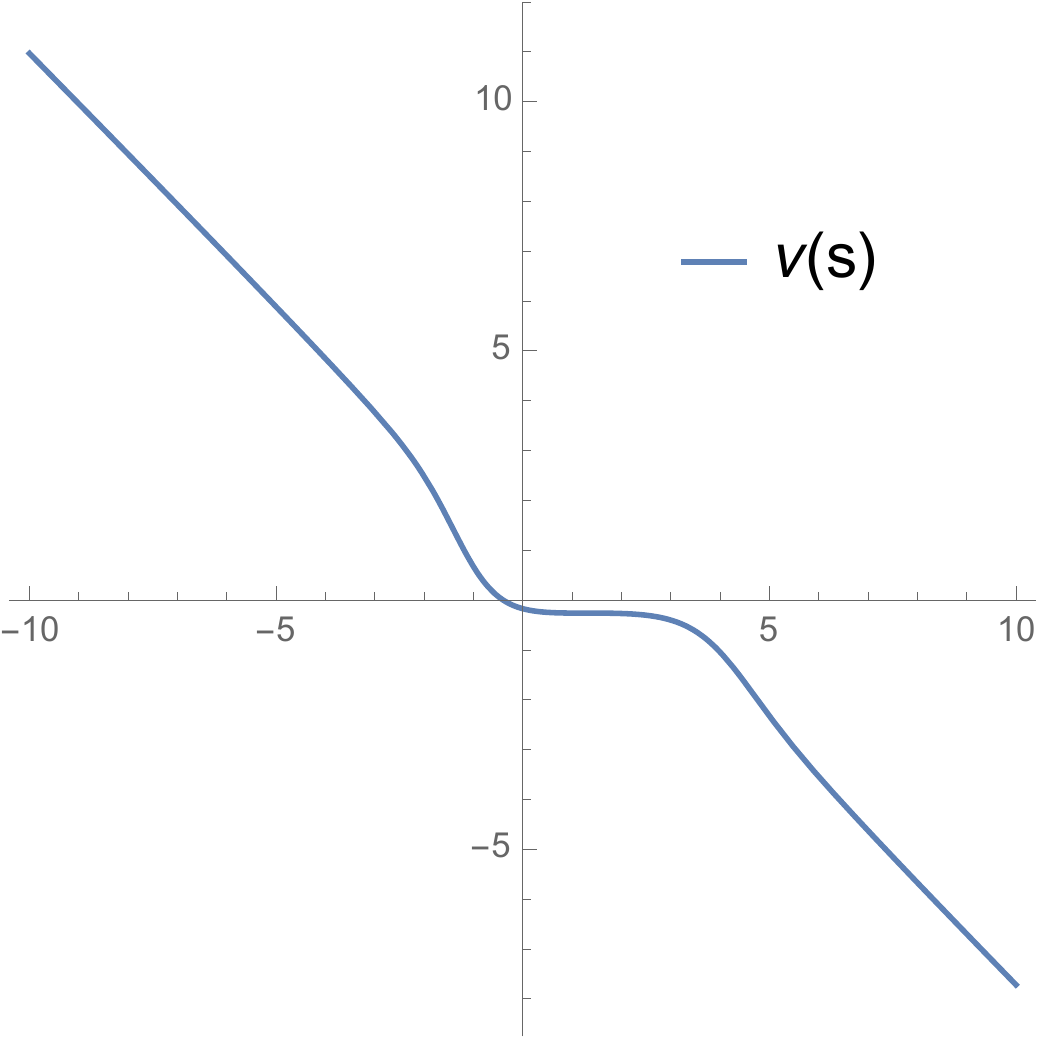}}
   \resizebox{2cm}{!}{\includegraphics{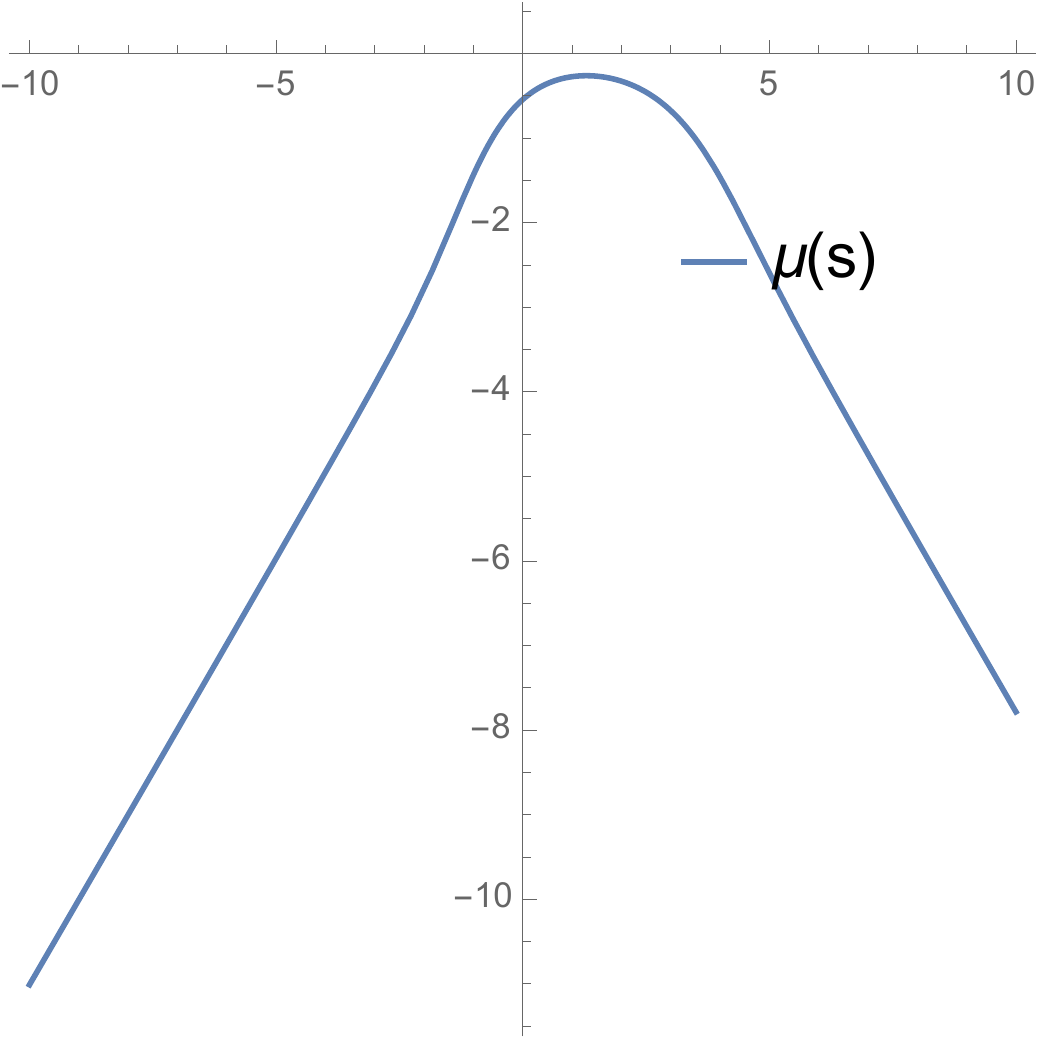}}
   \caption{Curve in the Poincar\'e disk model evolving by isometries with null ${\tilde v}$.}\label{figure7}
  \end{center}
\end{figure}

\subsection{Proof of the main theorem}

First we give a brief proof of the following simple property.

\begin{lemma}\label{lemma4.9}
Curves evolving by isometries in ${\mathbb H}^2$ have no self-intersections, and are properly embedded.
\end{lemma}

\begin{proof}
By way of contradiction, say that a curve $X(s)$ self-intersects at two parameter values $s_2>s_1$ with no self-intersections in between. At the self-intersection, the ingoing and outgoing tangents make angle $\alpha$. The region bounded by $X(s)$ for $s_1\le s\le s_2$ has area $A>0$. Since the Gauss curvature of ${\mathbb H}^2$ is $-1$, Gauss-Bonnet yields
\begin{equation}
\label{eq4.12}
\oint\limits_{s_1}^{s_2} \kappa_g ds=A+\pi+\alpha >\pi\ .
\end{equation}
But at a point of self-intersection we must have $\mu(s_1)=\mu(s_2)$, and so
\begin{equation}
\label{eq4.13}
0=\mu(s_2)-\mu(s_1)=\oint\limits_{s_1}^{s_2} \mu'(s) ds =\oint\limits_{s_1}^{s_2} \tau ds = \frac{1}{a} \oint\limits_{s_1}^{s_2} \kappa_g ds\ .
\end{equation}
Comparing \eqref{eq4.12} and \eqref{eq4.13}, we arrive at a contradiction, so there are no self-intersections. Then since every curve in Lemmata \ref{lemma4.6}--\ref{lemma4.8} escapes any bounded set, there are no accumulation points either, so the curves are properly embedded.
\end{proof}

The main theorem now follows easily from the above results.

\begin{proof}[Proof of Theorem \ref{theorem1.2}] Choose a vector ${\tilde v}$. If it is not null, write ${\tilde v}=av$ such that $v$ is normalized and define the variables $\tau$, $\nu$, and $\mu$ along  smooth curves $X(s)$ using \eqref{eq2.15}. If we require that the system \eqref{eq2.18} holds then from Proposition \ref{proposition2.2}, each solution of \eqref{eq2.18} determines a curve $X(s)$. Solutions of \eqref{eq2.18} are parametrized by the initial data $(\tau_0,\nu_0,\mu_0)$ for \eqref{eq2.18}, subject to the constraint \eqref{eq2.20} at $s=0$. Hence we obtain $2$-parameter families of solutions of \eqref{eq2.18} and of the associated curves $X(s)$. The solution curves have domain $s\in{\mathbb R}$, for $s$ a unit speed parameter, and $\mu=\langle X,v\rangle$ is divergent at least at one end, so the solutions are complete and noncompact.

If $\mu$ converges at one end, then by Lemma \ref{lemma4.1} we have $\mu\to 0$ while by Lemma \ref{lemma4.5} we also have $\tau\to 0$ and thus by Proposition \ref{proposition2.3} then $\kappa_g\to 0$. In this case $X(s)$ approaches a geodesic at this end. Lemmata \ref{lemma4.6}--\ref{lemma4.8} give conditions under which this situation does not arise.

If $\mu$ diverges, invoking Lemma \ref{lemma4.5} for $s\to\infty$, or for $s\to -\infty$ as the case may be, we have that the curvature obeys $\kappa_g\to \pm \frac{1}{a}\neq 0$, and so $X$ approaches a horosphere.

Lemma \ref{lemma4.9} establishes that $X$ is properly embedded.
\end{proof}


\begin{thebibliography}{99}
\bibitem{CZ} K-S Chou and X-P Zhu, \emph{The curve shortening problem}, (CRC Press, Boca Raton, 2001).
\bibitem{Gage1} M Gage, \emph{An isoperimetric inequality with application to curve shortening}, Duke Math J 50
(1983, 1225--1229.
\bibitem{Gage2} M Gage, \emph{Curve shortening makes convex curves circular}, Invent Math 76 (1984) 357--364.
\bibitem{Gage3} ME Gage, \emph{Curve shortening on surfaces}, Ann Sci l'\'ENS $4^e$ s\'erie, 23 (1990) 229--256.
\bibitem{GH} M Gage and RS Hamilton, \emph{The heat equation shrinking convex plane curves}, J Diff Geom 23 (1986) 69--96.
\bibitem{Grayson} MA Grayson, \emph{Shortening embedded curves}, Ann Math 129 (1989) 71--111.
\bibitem{Halldorsson1} HP Halld\'orsson, \emph{Self-similar solutions to the mean curvature flow in Euclidean and Minkowski space}, PhD thesis, Massachusetts Institute of Technology, 2013 (unpublished).
\bibitem{Halldorsson2} HP Halld\'orsson, \emph{Self-similar solutions to the curve shortening flow}, Trans Amer Math Soc 364 (2012) 5285--5309.
\bibitem{Kibble} TWB Kibble, \emph{Topology of cosmic domains and strings}, J Phys A: Math Gen 9 (1976) 1387--1398.
\bibitem{daSilva} FN da Silva, \emph{Fluxos de curvas no espaco hiperb\'olico e no cone de luz}, University of Brasilia PhD thesis, (2020, unpublished).
\bibitem{ST} FN da Silva and K Tenenblat, \emph{Soliton solutions to the curve shortening flow on the 2-dimensional hyperbolic plane}, preprint [arxiv:2102.07916v2].
\bibitem{SdRT} HF Santos Dos Reis and K Tenenblat, \emph{Soliton solutions to the curve shortening flow on the sphere}, Proc Amer Math Soc 147 (2019) 11, 4955-4967.
\end{thebibliography}
\end{document}